\documentclass[twoside,11pt]{article}

\usepackage{jmlr2e}

\usepackage{times}
\usepackage{graphicx}
\usepackage{subfigure} 

\usepackage{natbib}

\usepackage{algorithm}
\usepackage{algorithmic}

\usepackage{amsmath}
\usepackage{amssymb}
\usepackage{enumitem}
\usepackage{multirow}

\usepackage{color}

\usepackage{array}
\newcolumntype{L}[1]{>{\raggedright\let\newline\\\arraybackslash\hspace{0pt}}m{#1}}
\newcolumntype{C}[1]{>{\centering\let\newline  \\\arraybackslash\hspace{0pt}}m{#1}}
\newcolumntype{R}[1]{>{\raggedleft\let\newline \\\arraybackslash\hspace{0pt}}m{#1}}

\def\la{\langle}
\def\ra{\rangle}

\newcommand{\Px}[2]{\text{prox}_{#1}(#2)}
\newcommand{\NM}[2]{\| #1 \|_{#2}}
\newcommand{\R}{\mathbb{R}}

\newcommand{\Diag}[1]{\text{Diag}(#1) }

\newcommand{\TV}[1]{\text{TV}(#1)}
\newcommand{\st}{\text{\;s.t.\;}}
\newcommand{\Tr}[1]{\text{\,Tr}\left(#1\right)}
\newcommand{\sym}[1]{\text{\,sym}(#1)}

\jmlrheading{---}{---}{---}{02/17}{---}{Quanming Yao and James T. Kwok}

\ShortHeadings{Efficient Learning of Nonconvex Regularizers by Redistributing Nonconvexity}{Yao and Kwok}
\firstpageno{1}

\begin{document}

\title{Efficient Learning with a Family of Nonconvex Regularizers by Redistributing Nonconvexity}

\author{\name Quanming Yao \email qyaoaa@cse.ust.hk 
	\\
	\name James T. Kwok \email jamesk@cse.ust.hk
	\\
	\addr Department of Computer Science and Engineering \\
		Hong Kong University of Science and Technology \\
		Hong Kong		
}

\editor{}

\maketitle

\begin{abstract}
The use of convex regularizers allows for easy optimization, 
though they often produce biased estimation and inferior prediction performance. 
Recently, nonconvex regularizers have attracted a lot of attention and outperformed convex ones.  
However, the resultant optimization problem is much harder. 
In this paper, 
for a large class of nonconvex regularizers, 
we propose to move the nonconvexity from the regularizer to the loss. 
The nonconvex regularizer is then transformed to a familiar convex regularizer, 
while the resultant loss function can still be guaranteed to be smooth. 
Learning with the convexified regularizer can be performed by existing efficient 
algorithms originally designed for convex regularizers (such as the proximal algorithm, Frank-Wolfe algorithm,
alternating direction method of multipliers and stochastic gradient descent).
Extensions are made when
the convexified regularizer
does not have closed-form proximal step,
and when the loss function is nonconvex, nonsmooth.  Extensive experiments on a variety
of machine learning application scenarios show 
that
optimizing the transformed problem is much faster than running the state-of-the-art on the original problem.
\end{abstract}

\begin{keywords}
Nonconvex optimization, 
Nonconvex regularization,
Proximal algorithm,
Frank-Wolfe algorithm,
Matrix completion
\end{keywords}


\section{Introduction}
\label{sec:intro}

Risk minimization is fundamental to machine learning.
It admits a tradeoff between the empirical loss and regularization
as:
\begin{align}
\min_x F(x) \equiv f(x) + g(x),
\label{eq:pro}
\end{align}
where $x$ is the model parameter, $f$ is the loss and $g$ is the regularizer.
The choice of regularizers is important and 
application-specific, and is often the crux to obtain good prediction performance.
Popular examples include the sparsity-inducing regularizers, which have been commonly used in 
image processing \citep{beck2009fasttv,mairal2009online,jenatton2011proximal} and high-dimensional feature selection 
\citep{tibshirani2005sparsity,jacob2009group,liu2010moreau}; and
the low-rank regularizer in
matrix and tensor
learning, with
good empirical performance 
on tasks such as
recommender systems \citep{candes2009exact,mazumder2010spectral} and visual data analysis \citep{liu2013tensor,lu2014generalized}.

Most of these regularizers are convex. Well-known 
examples include the $\ell_1$-regularizer for sparse coding \citep{donoho2006compressed}, and
the nuclear norm regularizer in low-rank matrix learning \citep{candes2009exact}.
Besides having nice theoretical guarantees, convex regularizers also allow easy optimization.
Popular optimization algorithms in machine learning include the proximal algorithm \citep{parikh2013proximal},
Frank-Wolfe (FW) algorithm \citep{jaggi2013revisiting}, the alternating direction method of multipliers (ADMM)
\citep{boyd2011distributed},
stochastic gradient descent and its variants \citep{bottou1998online,xiao2014proximal}.
Many of these are
efficient, scalable, and have
sound convergence properties.

\begin{table*}[t]
	\centering
	\caption{Example nonconvex regularizers.
		Here, $\beta > 0$
		and $\theta > 0$.}
	\begin{tabular}{C{95px}|c | c |  c | c}
		\hline
		& $\kappa(\alpha)$                  
		& $\kappa'(\alpha)$                                                                            
		& $\kappa_0$ 
		& $\rho$\\ \hline
		GP \citep{geman1995nonlinear}
		& $\frac{\beta \alpha}{\theta + \alpha}$   
		& $\frac{\beta \theta}{(\theta + \alpha)^2}$    
		& $\frac{\beta}{\theta}$
		& $\frac{2\beta}{\theta^2}$
		\\ \hline
		LSP \citep{candes2008enhancing}
		& $\beta \log(1 + \frac{\alpha}{\theta})$ 
		& $\frac{\beta}{\theta + \alpha}$    
		& $\frac{\beta}{\theta}$
		& $\frac{\beta}{\theta^2}$ 
		\\ \hline
		MCP \citep{zhang2010nearly}  
		& $\begin{cases}
		\beta \alpha - \frac{\alpha^2}{2 \theta} & \alpha \le \beta \theta \\
		\frac{1}{2}\theta \beta^2          & \alpha > \beta \theta
		\end{cases}$                                       
		
		& $\begin{cases}
		\beta - \frac{\alpha}{\theta} & \alpha \le \beta \theta \\
		0         & \alpha > \beta \theta
		\end{cases}$                                                                                
		
		& $\beta$  
		& $\frac{1}{\theta}$ 
		\\ \hline
		Laplace \citep{trzasko2009highly}
		&  $\beta(1 - \exp(-\frac{\alpha}{\theta}))$                                      
		&  $\frac{\beta}{\theta} \exp\left(-\frac{\alpha}{\theta}\right) $                                                                                          
		& $\frac{\beta}{\theta}$
		& $\frac{\beta}{\theta^2}$ \\ \hline
		SCAD \citep{fan2001variable}    
		&   $\begin{cases}
		\beta \alpha & \alpha \le \beta \\
		\frac{- \alpha^2 + 2\theta\beta \alpha - \beta^2}{2(\theta - 1)} & \beta < \alpha \le \theta\beta \\
		\frac{\beta^2(1+\theta)}{2} & \alpha > \theta\beta
		\end{cases}$                                     
		
		& $\begin{cases}
		\beta                                     & \alpha \le \beta             \\
		\frac{-\alpha + \theta\beta}{\theta - 1} & \beta < \alpha \le \theta\beta \\
		0   & \alpha > \theta\beta
		\end{cases}$                                                                                
		
		& $\beta$
		
		& $\frac{1}{\theta - 1}$ 
		\\ \hline
	\end{tabular}
	\label{tab:regdecomp}
\end{table*}

However, convex regularizers often lead to biased
estimation.  For example, in sparse coding,  the solution obtained by the $\ell_1$-regularizer
is often not as sparse and accurate \citep{zhang2010analysis}.
In low-rank matrix learning,
the estimated rank 
obtained with the nuclear norm regularizer
is often much higher 
\citep{mazumder2010spectral}.
To alleviate this problem,
a number of nonconvex regularizers have been recently proposed
\citep{geman1995nonlinear,fan2001variable,candes2008enhancing,zhang2010nearly,trzasko2009highly}.
As can be seen from Table~\ref{tab:regdecomp},
they are all 
(i) nonsmooth at zero, which encourage a sparse solution; and (ii) concave, 
which place a smaller penalty than the $\ell_1$-regularizer
on features with large magnitudes.
Empirically, these nonconvex regularizers
usually 
outperform convex regularizers.

Even with a convex loss, the resulting nonconvex problem is much harder to
optimize. One can use
general-purpose nonconvex optimization solvers such as 
the concave-convex procedure \citep{yuille2002concave}. However,
the subproblem 
in each iteration
can be as expensive as the original problem,
and 
the concave-convex procedure 
is thus often slow in practice \citep{gongZLHY2013,zhongK2014gdpan}.

Recently, the proximal algorithm has also been extended for nonconvex problems.
Examples include the NIPS \citep{sra2012scalable}, IPiano \citep{ochs2014ipiano},
UAG \citep{ghadimi2016accelerated}, 
GIST \citep{gongZLHY2013},
IFB \citep{boct2016inertial}, and nmAPG \citep{li2015accelerated}.
Specifically, NIPS, IPiano and UAG allow 
$f$ 
in (\ref{eq:pro})
to be Lipschitz smooth (possibly nonconvex) but $g$ has to be convex;
while
GIST, IFB and nmAPG 
further allow $g$ to be nonconvex.
The current state-of-the-art is nmAPG.
However, efficient computation of
the underlying proximal operator is only possible
for simple nonconvex regularizers.
When the regularizer is complicated, 
such as the nonconvex versions of the fused lasso and overlapping group lasso regularizers \citep{zhongK2014gdpan},
the corresponding proximal step
has to be solved numerically
and is again expensive.
Another approach is by using
the proximal average
\citep{zhongK2014gdpan},
which computes and averages the
proximal step 
of each underlying regularizer.
However, because 
the proximal step
is only approximate,
convergence is usually slower than typical applications of the proximal algorithm \citep{li2015accelerated}.

When $f$ is smooth, there are endeavors to extend other algorithms from convex to nonconvex optimization.
For the global consensus problem, standard ADMM converges only when $g$ is convex 
\citep{hong2016convergence}.
When $g$ is nonconvex,
convergence of ADMM is only established for problems of the form $\min_{x, y}f(x) + g(y): y = A x$, 
where matrix $A$ has full row rank \citep{li2015global}.
The convergence of ADMM in more general cases is an open issue.
More recently, the
stochastic variance reduced gradient (SVRG) algorithm \citep{johnson2013accelerating},
which is a variant of the popular stochastic gradient descent
with reduced variance in the gradient estimates,
has also been extended for problems 
with
nonconvex 
$f$.
However, the regularizer $g$ is still required to be convex
\citep{reddiHSPS2016,zhuH2016}.

Sometimes, it is desirable to have a 
nonsmooth loss $f$.
For example, the absolute loss is more robust to outliers than the square loss, and
has been popularly used in applications such as  
image denoising 
\citep{yan2013restoration}, 
robust dictionary learning 
\citep{zhao2011background}
and robust PCA 
\citep{candes2011robust}.
The resulting optimization problem becomes more challenging.
When both $f$ and $g$ are convex,
ADMM is often the main optimization tool for problem~\eqref{eq:pro} \citep{he20121}.
However, when either $f$ or $g$ is nonconvex, ADMM no longer guarantees convergence. 
Besides a nonconvex $g$,
we may also want to use a nonconvex loss $f$,
such as $\ell_0$-norm \citep{yan2013restoration} 
and capped-$\ell_1$ norm \citep{sun2013robust},
as they are more robust to outliers
and can 
obtain
better performance.
However, when 
$f$ is nonsmooth and nonconvex, 
none of the above-mentioned algorithms (i.e., proximal algorithms,
FW algorithms, ADMM, and SVRG)  can be used.
As a last resort, one can use more general nonconvex optimization approaches
such as convex concave programming (CCCP) \citep{yuille2002concave}.
However, they are slow in general.

In this paper, we first consider the case where the
loss function $f$ is smooth (possibly nonconvex) and the regularizer $g$ is nonconvex.
We propose to handle nonconvex regularizers by reusing the
abundant repository of efficient 
convex algorithms originally designed for convex regularizers.
The key is to shift the nonconvexity associated with the nonconvex regularizer to the loss
function, and transform the nonconvex regularizer to a familiar convex regularizer.
To illustrate the practical usefulness of this convexification scheme, we show how it can be used
with popular optimization algorithms in machine learning.
For example, for the proximal algorithm,
the resultant proximal step can be much easier after transformation.
Specifically,
for the nonconvex tree-structured lasso and nonconvex sparse group lasso,
we show that the corresponding proximal steps have closed-form solutions on the
transformed problems, but not on the original ones.
For the nonconvex total variation problem,
though there is no closed-form solution for the proximal step before and after the transformation,
we show that the proximal step is still cheaper and easier for optimization after the transformation.
To allow further speedup,
we propose a proximal algorithm variant that allows the use of inexact proximal
steps with convex $g$ 
when it has no closed-form proximal step solution.
For the FW algorithm,
we consider its application to nonconvex low-rank matrix learning problems,
and propose a variant with guaranteed convergence to a critical point of
the nonconvex problem.
For SVRG in stochastic optimization 
and ADMM in consensus optimization,
we show that these algorithms have convergence guarantees on the transformed problems
but not on the original ones.

We further consider the case where $f$ is also nonconvex and nonsmooth (and 
$g$ is nonconvex).
We demonstrate that 
problem \eqref{eq:pro} can be transformed to an equivalent problem with a
smooth loss and convex regularizer using our proposed idea.
However, as the proximal step with the transformed regularizer has to be solved numerically
and exact proximal step is required,
usage with the proximal algorithm
may not be efficient.
We show that this problem can be addressed by the proposed inexact proximal algorithm.
Finally,
in the experiments,
we demonstrate the above-mentioned advantages of optimizing the transformed problems instead of the original ones on various tasks,
and show  that running algorithms on the transformed problems can be much faster than the state-of-art on the original ones.

The rest of the paper is organized as follows. Section~\ref{sec:relworks}
provides a review on the related works. 
The main idea for problem transformation is presented in Section~\ref{sec:shift},
and its usage with various algorithms are discussed in Section~\ref{sec:examples}. 
Experimental results are shown in Section~\ref{sec:expt},
and the last section gives some concluding remarks.
All the proofs are in Appendix~\ref{sec:proof}.
Note that this paper extends a shorter version published in the proceedings of the International
Conference of Machine Learning \citep{yao2016efficient}.


\subsection*{Notation}

We denote vectors and matrices
by lowercase and uppercase boldface letters, respectively.
For a vector $x \in \R^d$,
$\NM{x}{2} = (\sum_{i = 1}^d |x_i|^2)^{1/2}$ is its $\ell_2$-norm,
$\Diag{x}$ returns a diagonal matrix $X \in \R^{d \times d}$ with $X_{ii} = x_i$.
For a matrix $X \in \R^{m \times n}$ (where $m \le n$ without loss of generality),
its nuclear norm is
$\NM{X}{*} = \sum_{i = 1}^m \sigma_i(X)$, where $\sigma_i(X)$'s are 
the singular values of
$X$, and 
its Frobenius norm
is 
$\NM{X}{F} = \sqrt{\sum_{i = 1}^m \sum_{j = 1}^n X_{ij}^2}$,
and $\NM{X}{\infty} = \max_{i,j} |X_{ij}|$.
For a square matrix $X$, 
	$X \in \mathcal{S}_+$ indicates it is a positive semidefinite.
For two matrices $X$ and $Y$, $\langle X, Y \rangle = \sum_{i,j}
X_{ij} Y_{ij}$.
For a smooth function $f$, $\nabla f(x)$ is its gradient at $x$.
For a convex but nonsmooth $f$, 
$\partial f(x) = \{u: f(y) \ge f(x) + \la u,y - x\ra\}$ is its subdifferential at $x$,
and $g\in \partial f(x)$ is a subgradient.


\section{Related Works}
\label{sec:relworks}

In this section,
we review some popular algorithms for solving \eqref{eq:pro}. Here,
$f$ is assumed to be Lipschitz smooth.


\subsection{Convex-Concave Procedure (CCCP)}
\label{sec:difcp}

The convex-concave procedure (CCCP) \citep{yuille2002concave,zhaosong2012} is a popular 
and general solver for \eqref{eq:pro}.
It assumes that $F$ can be decomposed as a difference of convex (DC) functions \citep{hiriart85}, 
i.e., $F(x) = \tilde{F}(x) + \hat{F}(x)$
where $\tilde{F}$ is convex and $\hat{F}$ is concave. 
In each CCCP iteration, $\hat{F}$ is linearized  at $x_t$,
and $x_{t + 1}$ is generated as
\begin{align}
x_{t + 1} = 
\arg\min_{x}
\tilde{F}(x) + \hat{F}(x_{t}) - (x - x_t)^{\top} s_t,
\label{eq:dcsub}
\end{align}
where $s_t \in \partial [ - \hat{F}(x_t) ]$ is a subgradient.
Note that as the last two terms are linear, 
\eqref{eq:dcsub} is a convex problem and can be easier than the original problem $F$.

However, CCCP is expensive as \eqref{eq:dcsub} needs to be exactly solved.
Sequential convex programming (SCP) \citep{zhaosong2012} improves its efficiency when $F$ is in form of \eqref{eq:pro}.
It assumes that $f$ is $L$-Lipschitz smooth (possibly nonconvex); while $g$ can be nonconvex,
but admits a DC decomposition as $g(x) = \tilde{\varsigma}(x) + \hat{\varsigma}(x)$. 
It 
then
generates $x_{t + 1}$ as
\begin{align}
x_{t + 1} 
& = \arg\min_{x}
f(x_{t}) + (x - x_t)^{\top} \nabla f(x_t) + \frac{L}{2}\NM{x - x_t}{2}^2
+ \tilde{\varsigma}(x) + \hat{\varsigma}(x_t) 
- (x - x_t)^{\top} s_t
\notag \\
& = \arg\min_{x} \frac{1}{2}\NM{x - x_t - s_t + \frac{1}{L} \nabla f(x_t)}{2}^2 + \tilde{\varsigma}(x),
\label{eq:scpsub}
\end{align}
where $s_t \in \partial \left( -\hat{\varsigma}(x_t) \right)$. 
When $\tilde{\varsigma}$ is simple, \eqref{eq:scpsub} has a closed-form solution,
and 
SCP can be faster than CCCP.
However, its convergence is still slow 
in general \citep{gongZLHY2013,zhongK2014gdpan,li2015accelerated}.


\subsection{Proximal Algorithm}

The proximal algorithm \citep{parikh2013proximal}
has been popularly used for optimization problems of the form in \eqref{eq:pro}.
Let $f$ be convex and $L$-Lipschitz smooth, and
$g$ is convex.
The proximal algorithm 
generates iterates $\{ x_t \}$ as
\begin{align*}
x_{t + 1} 
& = \arg\min_{x} f(x_t) + (x - x_t)^{\top} \nabla f(x_t) + \frac{L}{2}\NM{x - x_t}{2}^2 + g(x) 
\\
& = \text{prox}_{\frac{1}{L} g}\left(x_t - \frac{1}{L} \nabla f(x_t)\right),
\end{align*}
where
$\Px{g}{z} \equiv \arg\min_{x} \frac{1}{2}\NM{x - z}{2}^2 + g(x)$
is the proximal step,
The proximal algorithm converges at a rate of $O(1/T)$. This can be further accelerated to $O(1/T^2)$ by 
modifying the 
generation
of $\{ x_t \}$ as
\citep{beck2009fast,nesterov2013gradient}:
\begin{align*}
y_t 
& = x_t + \frac{\alpha_{t - 1} - 1}{\alpha_t}(x_t - x_{t - 1}),
\\
x_{t + 1} 
& = \text{prox}_{\frac{1}{L} g}\left(y_t - \frac{1}{L} \nabla f(y_t)\right),
\end{align*}
where $\alpha_0 = \alpha_1 = 1$ and $\alpha_{t + 1} = \frac{1}{2}(\sqrt{4 \alpha_t^2 + 1} + 1)$.

Recently, the proximal algorithm has been extended to 
nonconvex optimization.
In particular, NIPS \citep{sra2012scalable}, IPiano \citep{ochs2014ipiano}
and UAG \citep{ghadimi2016accelerated} allow $f$ to be nonconvex, while
$g$ is still required to be convex.
GIST \citep{gongZLHY2013},
IFB \citep{boct2016inertial} and nmAPG \citep{li2015accelerated}
further remove this restriction and allow $g$ to be nonconvex.
It is desirable that the proximal step has a closed-form solution.
This is true
for many convex regularizers
such as the lasso regularier \citep{tibshirani1996regression}, 
tree-structured lasso regularizer \citep{liu2010moreau,jenatton2011proximal} 
and sparse group lasso regularizer \citep{jacob2009group}.
However,
when $g$ is nonconvex, such solution only exists for some simple $g$,
e.g., nonconvex lasso regularizer \citep{gongZLHY2013},
and usually do not exist for more general cases,
e.g., nonconvex tree-structured lasso regularizer \citep{zhongK2014gdpan}.

On the other hand, 
\cite{zhongK2014gdpan} used proximal average \citep{bauschke2008proximal} to 
handle complicate $g$ which is in the form $g(x) = \sum_{i = 1}^K \mu_i g_i(x)$, where each $g_i$ has a simple proximal step.
The iterates are 
generated
as
\[ x_{t + 1} = 
\sum_{i = 1}^K \mu_i \cdot \text{prox}_{\frac{\mu_i}{L}g_i}\left(x_t - \frac{1}{L} \nabla
f(x_t)\right)/
\sum_{i = 1}^K \mu_i. \]
Each of the constituent proximal steps
$\text{prox}_{\frac{\mu_i}{L}g_i}(\cdot)$ can be computed inexpensively, and thus
the per-iteration complexity is low.
It only converges to an approximate solution to $\Px{g}{z}$, but an approximation guarantee is
provided.  However, empirically, the convergence can be slow.


\subsection{Frank-Wolfe (FW) Algorithm}
\label{sec:fw}

The FW algorithm \citep{frank1956algorithm}
is used for solving optimization problems of the form 
\begin{align}
\min_{x} f(x) \;:\; x \in \mathcal{C},
\label{eq:FW}
\end{align}
where $f$ is Lipschitz-smooth and convex, 
and $\mathcal{C}$ is a compact convex set.
Recently, it 
has been popularly used 
in machine learning 
\citep{jaggi2013revisiting}.
In each iteration, the FW algorithm generates the next iterate $ x_{t+1}$ as
\begin{eqnarray}
s_t & = & \arg\min_{s \in \mathcal{C}} \; s^{\top} \nabla f(x_t), \label{eq:stdfw1} \\
\gamma_t & = & \arg\min_{\gamma \in [0,1]} f( (1 - \gamma)x_t + \gamma s_t ), \label{eq:stdfw3} \\
x_{t + 1} & = & (1 - \gamma_t) x_t + \gamma_t s_t.\label{eq:stdfw4} 
\end{eqnarray}
Here,
(\ref{eq:stdfw1}) is a linear subproblem which can often be easily solved; 
(\ref{eq:stdfw3}) performs line search, 
and the next iterate $x_{t + 1}$ is generated from a 
convex combination of $x_t$ and $s_t$ in 
(\ref{eq:stdfw4}).
The FW algorithm has a 
convergence rate
of $O(1/T)$ \citep{jaggi2013revisiting}.

In this paper,
we will focus on using the FW algorithm to learn a low-rank matrix
$X\in \R^{m\times n}$.
Without loss of generality, we assume that $m \le n$.
Let $\sigma_i(X)$'s be the singular values of $X$.
The nuclear norm of $X$, $\NM{X}{*} = \sum_{i = 1}^m \sigma_i(X)$,
is the tightest convex envelope of $\text{rank}(X)$, and is often used as a low-rank regularizer \citep{candes2009exact}. 
The low-rank matrix learning problem can be written as
\begin{align} 
\label{eq:lowrank}
\min_{X} f(X) + \mu \NM{X}{*},
\end{align}
where $f$ is the loss.
For example,
in matrix completion
\citep{candes2009exact},
\begin{equation} \label{eq:loss}
f(X)= \frac{1}{2}\NM{\mathcal{P}_{\Omega}(X - O)}{F}^2,
\end{equation} 
where $O$ is the observed incomplete matrix, 
$\Omega \in \{ 0, 1 \}^{m \times n}$ contains indices
to the observed entries in $O$,
and
$[P_\Omega(A)]_{ij} = A_{ij}$ if $\Omega_{ij} = 1$,
and  0 otherwise.


The FW algorithm
for this nuclear norm regularized problem is shown in Algorithm~\ref{alg:stdfw}
\citep{zhang2012accelerated}.
Let the iterate 
at the $t$th iteration be $X_t$.
As in \eqref{eq:stdfw1},
the following linear subproblem has to be solved 
\citep{jaggi2013revisiting}:
\begin{equation} \label{eq:subprob}
\min_{S:\NM{S}{*} \le 1} \langle S, \nabla f(X_t) \rangle. 
\end{equation} 
This can be obtained from the rank-one SVD of $\nabla f(X_t)$
(step 3).
Similar to \eqref{eq:stdfw3},  line search is performed at step~4. 
As a rank-one matrix is added into $X_t$ in each iteration, it is convenient to
write $X_t$ as 
\begin{equation} \label{eq:uv}
\sum_{i = 1}^t u_i v_i^{\top} = U_t V_t^{\top}, 
\end{equation} 
where $U_t = [u_1, \dots, u_t]$ and $V_t= [v_1, \dots, v_t]$.
The FW algorithm has a convergence 
rate 
of $O(1/T)$ 
\citep{jaggi2013revisiting}.
To make it empirically faster,
Algorithm~\ref{alg:stdfw} also performs optimization at step~6 \citep{laue2012hybrid,zhang2012accelerated}.
Substituting $\NM{X}{*} = \min_{X = U V^{\top}} \frac{1}{2}\left( \NM{U}{F}^2 + \NM{V}{F}^2 \right)$ \citep{srebro2004maximum} into \eqref{eq:lowrank}, 
we have the following local optimization problem:
\begin{align}
\label{eq:fwfactor}
\min_{U, V} f(U V^{\top}) 
+ \frac{\mu}{2}( \NM{U}{F}^2 + \NM{V}{F}^2 ).
\end{align}
This can be solved by standard solvers such as L-BFGS \citep{nocedal2006numerical}.

\begin{algorithm}[ht]
\caption{Frank-Wolfe algorithm for problem~\eqref{eq:lowrank} with $f$ convex \citep{zhang2012accelerated}.}
	\begin{algorithmic}[1]
		\STATE $U_1 = [\;]$ and $V_1 = [\; ]$;
		\FOR{$t = 1 \dots T$}
		\STATE $[u_t, s_t, v_t] = \text{rank1SVD}(\nabla f(X_t))$;
		\STATE $[\alpha_t, \beta_t] = \arg\min_{\alpha \ge 0,
			\beta \ge 0} f( \alpha X_t + \beta u_t v_t^{\top} )
		+ \mu (\alpha\NM{X_t}{*} + \beta) $;
		\STATE $\bar{U}_t = \left[ \sqrt{\alpha_t} U_t; \sqrt{\beta_t} u_t \right]$ and
		$\bar{V}_t = \left[ \sqrt{\alpha_t} V_t; \sqrt{\beta_t} v_t \right]$; 
		\STATE obtain $[U_{t + 1}, V_{t + 1}]$ from \eqref{eq:fwfactor}, using $\bar{U}_t$ and $\bar{V}_t$ for warm-start;
		// $X_{t + 1} = U_{t + 1} V_{t + 1}^{\top}$
		\ENDFOR
		\RETURN $U_{T + 1}$ and $V_{T + 1}$.
	\end{algorithmic}
	\label{alg:stdfw}
\end{algorithm}


\subsection{Alternating Direction Method of Multipliers (ADMM)}

ADMM is a simple but powerful algorithm first introduced in the 1970s
\citep{glowinski1975}.
Recently, 
it has 
been popularly used in diverse fields such as
machine learning, data mining and image processing \citep{boyd2011distributed}.
It can be used to solve optimization problems of the form
\begin{equation} \label{eq:admobj}
\min_{x,y} \; f(x)+g(y) \; : \; Ax+By = c,  
\end{equation}
where $f, g$ are convex functions, and $A,B$ (resp. $c$) are
constant matrices (resp. vector) of appropriate sizes.
Consider the augmented Lagrangian
$L(x,y,u) = f(x)+g(y)+u^\top(Ax+By-c) +\frac{\tau}{2}\NM{Ax+By-c}{2}^2$,
where $u$ is the vector of Lagrangian multipliers, and $\tau>0$ is a penalty
parameter.
At the $t$th iteration of ADMM,  the values of $x,y$ and $u$ 
are updated as
\begin{eqnarray}
x_{t+1} & = & \arg\min_{x} L(x,y_{t},u_{t}), 
\label{eq:ADMM1}\\
y_{t+1} & = & \arg\min_{y} L(x_{t+1},y,u_{t}), 
\label{eq:ADMM2}\\
u_{t+1} & = & u_{t} + \tau(Ax_{t+1}+By_{t+1}-c).
\notag
\end{eqnarray}
By minimizing $L(x,y,u_k)$ w.r.t.  $x$ and $y$ in an alternating manner ((\ref{eq:ADMM1}) and
(\ref{eq:ADMM2})),
ADMM can more easily decompose the optimization problem when $f,g$ are separable.

In this paper, we will focus a special case of (\ref{eq:admobj}), namely, the consensus optimization problem:
\begin{align}
\min_{y, x^1, \dots, x^M}
\sum_{i = 1}^M f_i(x^i) + g(y) \;\;:\;\;
x^1 = \dots = x^M = y,
\label{eq:ADMMpro2}
\end{align}
Here, each $f_i$ is Lipschitz-smooth, $x^i$ is the variable in the local objective $f_i$, and $y$ is the global consensus variable.
This type of problems is often
encountered in machine learning, signal
processing and wireless communication \citep{Bertsekas1989,boyd2011distributed}. 
For example, in regularized
risk minimization, $y$ is the model parameter, $f_i$
is the regularized risk functional defined
on data subset $i$, and $g$ is the 
regularizer.
When $f_i$ is smooth and $g$ is convex, ADMM converges to a critical point of 
\eqref{eq:ADMMpro2} \citep{hong2016convergence}.
However, when $g$ is nonconvex, its convergence is still an open issue.


\section{Shifting Nonconvexity from Regularizer to Loss}
\label{sec:shift}

In recent years, a number of nonconvex regularizers have been 
proposed.
Examples include
the Geman penalty (GP) \citep{geman1995nonlinear},
log-sum penalty (LSP) \citep{candes2008enhancing}
and
Laplace penalty \citep{trzasko2009highly}.
In general,
learning with nonconvex regularizers is much more difficult than learning with convex
regularizers.  In this section, we show how to move the nonconvex component from the nonconvex
regularizers to the loss function.  Existing algorithms can then be reused to learn with
the convexified regularizers.

First, we make the following standard assumptions  on \eqref{eq:pro}.

\begin{itemize}[itemsep = 0cm, topsep=0.125cm]
	\item[A1.] $F$ is bounded from below and $\lim_{\NM{x}{2} \rightarrow \infty} F(x) = \infty$;
	
\item[A2.] $f$ is $L$-Lipschitz smooth
(i.e., $\NM{\nabla f(x) - \nabla f(y)}{2} \le L \NM{x - y}{2}$), but possibly nonconvex.
\end{itemize}
 
Let $\kappa$ be a function that is concave, non-decreasing, 
$\rho$-Lipschitz smooth with $\kappa'$ non-differentiable at finite points, 
and $\kappa(0) = 0$.
With the exception
of the capped-$\ell_1$ norm penalty \citep{zhang2010nearly} and $\ell_0$-norm regularizer,
all regularizers in Table~\ref{tab:regdecomp} satisfy requirements on $\kappa$.
We consider $g$ of the following forms.
\begin{enumerate} 
	\item[{\bf C1.}] $g(x) = \sum_{i = 1}^K \mu_i g_i(x)$, where
	$\mu_i \ge 0$, 
	\begin{equation} 
	\label{eq:g}
	g_i(x) = \kappa(\NM{A_i x}{2}),
	\end{equation} 
	and $A_i$ is a matrix.
When $\kappa$ is the identity function, $g(x)$ reduces to the convex regularizer $\sum_{i = 1}^K \mu_i
\NM{A_i x}{2}$. By using different $A_i$'s,  $g$ becomes various structured sparsity regularizers such as the group lasso \citep{jacob2009group}, fused lasso \citep{tibshirani2005sparsity}, and graphical lasso \citep{jacob2009group}.
\item[{\bf C2.}] $g(X) = \mu \sum_{i = 1}^m \kappa(\sigma_i(X))$,
where $X$ is a matrix and
$\mu \ge 0$.
	When $\kappa$ is the identity function, 
	$g$ reduces to the nuclear norm.
\end{enumerate} 
First, 
consider $g$ in {\bf C1}. Rewrite each nonconvex $g_i$ in (\ref{eq:g})
as 
\begin{equation} \label{eq:gi}
g_i(x) = \bar{g}_i(x) + \kappa_0 \NM{A_i x}{2}, 
\end{equation} 
where $\kappa_0 =\kappa'(0)$, and
$\bar{g}_i(x) = \kappa(\NM{ A_i x}{2}) - \kappa_0 \NM{A_i x}{2}$.
Obviously,  $\kappa_0 \NM{A_i x}{2}$ 
is convex but nonsmooth.
The following shows that $\bar{g}_i$,
though nonconvex, is concave and Lipschitz smooth.
In the sequel, a function with a bar on top (e.g., $\bar{f}$) denotes that it is smooth; whereas a function with breve (e.g., $\breve{g}$) denotes that it may be nonsmooth.

\begin{proposition} 
\label{prop:smooth}
$\kappa(\NM{z}{2}) - \kappa_0 \NM{z}{2}$ 
is concave and 
$2\rho$-Lipschitz smooth.
\end{proposition}

\begin{corollary} 
\label{cor:smooth}
$\bar{g}_i$ is concave and Lipschitz smooth with modulus
$\bar{L}_i = 2 \rho \NM{A_i}{F}$.
\end{corollary} 

\begin{corollary} \label{cor:c1}
$g(x)$ can be decomposed as
$\bar{g}(x) + \breve{g}(x)$, where
$\bar{g}(x) \equiv \sum_{i = 1}^K \mu_i\bar{g}_i(x)$  is concave
and Lipschitz-smooth, while
$\breve{g}(x) \equiv \kappa_0 \sum_{i = 1}^K \mu_i \NM{A_i x}{2}$
is convex but nonsmooth.
\end{corollary} 

\begin{remark} \label{remark:c2}
When
$A_i = \Diag{e_i}$, where $e_i$ is the unit vector for dimension $i$, $\NM{A_i x}{2} = |x_i|$ and 
\begin{equation} \label{eq:c2}
g(x) = \sum_{i = 1}^d \mu_i \kappa(\NM{A_i x}{2}) = \sum_{i = 1}^d \mu_i \kappa(|x_i|). 
\end{equation} 
Using Corollary~\ref{cor:c1},
	$g$ can be decomposed as $\bar{g}(x)+\breve{g}(x)$, where 
	$\bar{g}(x) \equiv  \sum_{i = 1}^d \mu_i ( \kappa(|x_i|) - \kappa_0 |x_i| ) $
	is concave and $2\rho$-Lipschitz smooth,
	while $\breve{g}(x) \equiv \kappa_0 \sum_{i = 1}^d \mu_i |x_i|$ is convex and nonsmooth.
When $d=1$ and $\mu_1=1$,
an illustration
of $g(x)=\kappa(|x|)$, 
$\bar{g}(x)=\kappa(|x|) - \kappa_0 |x|$ 
and 
$\breve{g}(x)=\kappa_0 |x|$ 
for the various nonconvex regularizers
is shown in 
Figure~\ref{fig:deompreg}.
When $\kappa$ is the identity function and $\mu_1 = \dots = \mu_m = \mu$, $g$ in
(\ref{eq:c2}) reduces to the lasso regularizer $\mu \NM{x}{1}$.
\end{remark} 

\begin{figure*}[t]
	\centering
	\subfigure[GP.]
	{\includegraphics[width = 0.32\textwidth]{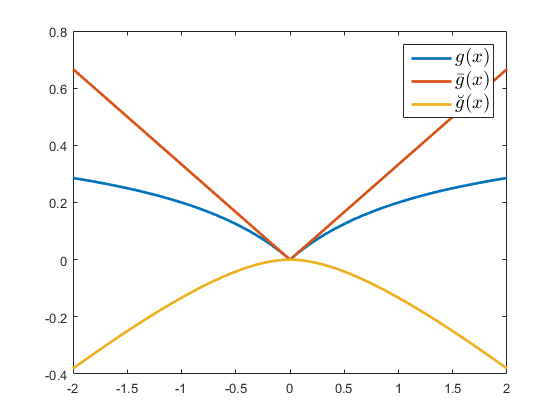}}
	\subfigure[LSP.]
	{\includegraphics[width = 0.32\textwidth]{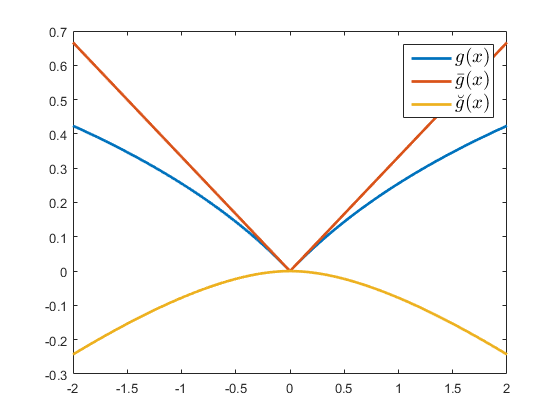}}
	\subfigure[MCP.]
	{\includegraphics[width = 0.32\textwidth]{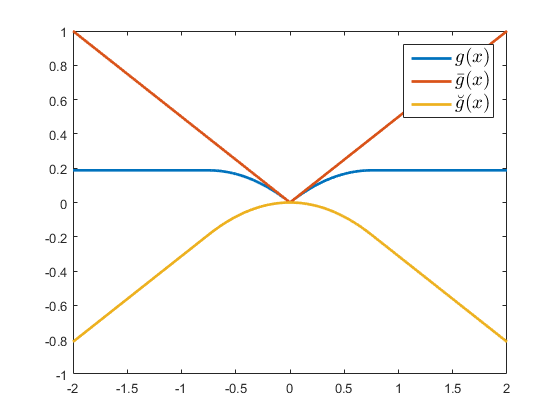}}
	\subfigure[Laplace.]
	{\includegraphics[width = 0.32\textwidth]{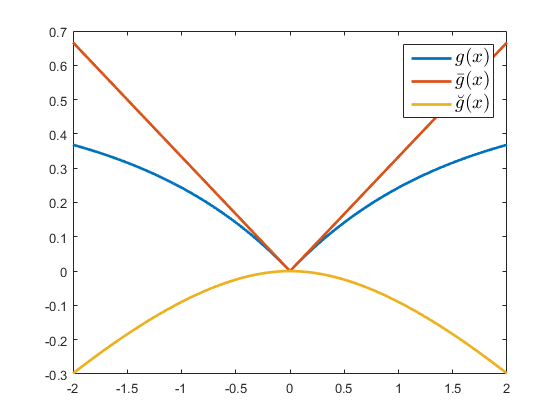}}
	\subfigure[SCAD.]
	{\includegraphics[width = 0.32\textwidth]{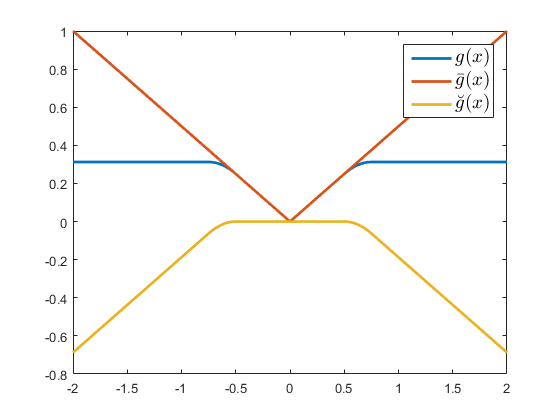}}
	\caption{Decompositions of the regularizers in Table~\ref{tab:regdecomp} for one-dimensional $z$
		($\beta = 0.5$, $\theta = 1.5$).}
	\label{fig:deompreg}
\end{figure*}

Using Corollary~\ref{cor:c1}, problem~\eqref{eq:pro} can then be rewritten as
\begin{equation} \label{eq:equiv}
\min_{x} \bar{f}(x) + \breve{g}(x),
\end{equation}
where $\bar{f}(x) \equiv f(x) + \bar{g}(x)$.
Note that $\bar{f}$ (which can be viewed as an augmented loss) is 
Lipschitz smooth
while $\breve{g}$ (viewed as a convexified regularizer) is convex but possibly nonsmooth.
In other words, nonconvexity is shifted from the regularizer $g$ to the loss $f$, 
while ensuring that the augmented loss is smooth. 


When $X$ is a matrix,
similar to Corollary~\ref{cor:c1},
the following Proposition~\ref{prop:c3} holds for $g$ in {\bf C2}.

\begin{proposition} \label{prop:c3}
Any $g$ in {\bf C2}
can be decomposed as $\bar{g}(X)+\breve{g}(X)$, where 
\begin{equation} \label{eq:barg}
\bar{g}(X) \equiv \mu \sum_{i = 1}^m \kappa(\sigma_i(X)) - \mu \kappa_0 \NM{X}{*}
\end{equation} 
is concave and $2\rho$-Lipschitz smooth, 
while $\breve{g}(X) \equiv \kappa_0 \NM{X}{*}$ is convex and nonsmooth.
\end{proposition} 

Since $\bar{g}$ is concave and $\breve{g}$ is convex,
the nonconvex regularizer $g
=
\breve{g}
-(-\bar{g})
$
can be viewed as
a difference of convex functions (DC) 
\citep{hiriart85}.
\cite{zhaosong2012,gongZLHY2013,zhongK2014gdpan}
also relied on DC decompositions
of the nonconvex regularizer.
However, they   do not utilize this in the computational procedures, while
we use the DC decomposition to simplify the regularizers.
As will be seen, though the DC decomposition of a nonconvex function is not unique in
general,
the particular one proposed here is crucial for
efficient optimization.


\section{Example Use Cases}
\label{sec:examples}

In this section, 
we provide concrete examples to show how the proposed convexification scheme can be 
used with various optimization algorithms.
An overview is summarized in Table~\ref{tab:advantages}.

\begin{table}[ht]
\centering
\caption{Using the proposed convexification scheme with various algorithms.}
\begin{tabular}{ c | c | c} \hline
& section                                   & advantages                                                  \\ \hline 
proximal algorithm & \ref{sec:withprox}, \ref{sec:nonsthmloss} & cheaper                               proximal step \\
FW algorithm       & \ref{sec:noncvxFW}                        & cheaper                          linear subproblem \\
	(consensus)   ADMM & \ref{sec:ADMM}                            & 
	                   cheaper                               proximal step; provide convergence guarantee                                   \\
	SVRG               & \ref{sec:svrg}                            & 
	                   cheaper                               proximal step; provide convergence guarantee                                   \\
	mOWL-QN            & \ref{sec:quasiNT}                         & simpler analysis; 
	capture curvature information \\ \hline
\end{tabular}
\label{tab:advantages}
\end{table}


\subsection{Proximal Algorithms}
\label{sec:withprox}

In this section, we provide example applications on using the
proximal algorithm for nonconvex
structured sparse learning. 
The proximal algorithm 
has been commonly used for learning with convex regularizers
\citep{parikh2013proximal}.
With a nonconvex regularizer, 
the underlying proximal step becomes much more challenging.
\cite{gongZLHY2013,li2015accelerated} and \cite{boct2016inertial}
extended proximal algorithm to simple nonconvex $g$, but
cannot handle more complicated nonconvex regularizers
such as the tree-structured lasso regularizer \citep{liu2010moreau,schmidt2011convergence}, 
sparse group lasso regularizer \citep{jacob2009group} and total variation regularizer \citep{nikolova2004variational}.
Using the proximal average \citep{bauschke2008proximal},
\cite{zhongK2014gdpan} 
can handle nonconvex regularizers of the form $g = \sum_{i = 1}^K \mu_i g_i$, where each $g_i$ is simple.
However, the solutions obtained are only approximate.
General nonconvex optimization techniques such as
the concave-convex procedure (CCCP) \citep{yuille2002concave} 
or its variant sequential convex programming (SCP) \citep{zhaosong2012}
can also be used, though
they are slow in general \citep{gongZLHY2013,zhongK2014gdpan}.

Using the proposed transformation, 
one only needs to solve the proximal step of a standard convex regularizer
instead of that of a nonconvex regularizer.
This allows reuse of existing solutions for the proximal step and is much less expensive.
As proximal algorithms have the same convergence guarantee for 
convex and nonconvex 
$f$ 
\citep{gongZLHY2013,li2015accelerated},
solving the transformed problem
can be much faster. The following gives some specific examples.


\subsubsection{Nonconvex Sparse Group Lasso}
\label{sec:gpslasso}

In sparse group lasso,
the feature vector $x$ is divided into groups. Assume that group
$\mathcal{G}_j$ contains dimensions in $x$
that group $j$ contains. Let
$\left[x_{\mathcal{G}_j}\right]_i = x_i$ if $i \in \mathcal{G}_j$, 
and 0 otherwise.
Given training samples $\{(a_1,y_1),\dots, (a_N,y_N) \}$,
(convex) sparse group lasso is formulated as 
\citep{jacob2009group}:
\begin{align}
\min_{x} \sum_{i = 1}^N \ell( y_i, a_i^{\top} x ) 
+ \lambda \NM{x}{1}
+ \sum_{j = 1}^{K} \mu_j \NM{x_{\mathcal{G}_j}}{2},
\label{eq:sglasso}
\end{align}
where 
$\ell$ is
a smooth loss, and $K$ is the number of 
(non-overlapping)
groups.

For the nonconvex extension, the regularizer becomes
\begin{equation} \label{eq:noncvx_gp}
g(x) = \lambda \sum_{i = 1}^d \kappa(|x_i|) + \sum_{j = 1}^{K}\mu_j \kappa( \NM{x_{\mathcal{G}_j}}{2} ).
\end{equation}
Using Corollary~\ref{cor:c1} and Remark~\ref{remark:c2}, 
the  convexified regularizer is  
$\breve{g}(x) = \kappa_0 ( \lambda \NM{x}{1} + \sum_{j = 1}^{K} \mu_j \NM{x_{\mathcal{G}_j}}{2} )$.
Its proximal step can be easily computed by the algorithm in 
\citep{yuan2011efficient}. Specifically,
the proximal operator of $\breve{g}$
can be obtained 
by computing
$\Px{\mu_j \NM{\cdot}{2}}{\Px{\lambda \NM{\cdot}{1}}{x_{\mathcal{G}_j}}}$
for each group separately.
This  can then be used with
any proximal algorithm that can handle nonconvex objectives
(as $\bar{f}$ is nonconvex).
In particular, we will adopt the state-of-the-art nonmontonic APG (nmAPG) algorithm
\citep{li2015accelerated} (shown in 
Algorithm~\ref{alg:nmAPG}).  Note that nmAPG cannot be directly used with the nonconvex regularizer $g$ in (\ref{eq:noncvx_gp}),
as the corresponding proximal step has no inexpensive closed-form solution.

\begin{algorithm}[ht]
\caption{Nonmonotonic APG (nmAPG) \citep{li2015accelerated}.}
	\begin{algorithmic}[1]
		\STATE Initialize $z_1 = x_1 = x_0$, $\alpha_0 = 0$, $\alpha_1 = 1$, $\eta \in [0,1)$, $c_1 = F(x_1)$, $q_1 = 1$,
		and stepsize $\tau > \bar{L}$, $\delta \in (0, \tau - \bar{L})$;
		\FOR{$t = 1, \dots, T$}
		\STATE $y_t = x_t + \frac{\alpha_{t - 1}}{\alpha_t}(z_t - x_t)
		+ \frac{\alpha_{t - 1} - 1}{\alpha_t}(x_t - x_{t - 1})$;
		\STATE $z_{t + 1} = \Px{\frac{1}{\tau}\breve{g}}{y_t - \frac{1}{\tau} \nabla \bar{f}(y_t)}$;
		\IF{$F(z_{t + 1}) \le c_t - \frac{\delta}{2}\NM{z_{t + 1} - y_t}{2}^2$}
		\STATE $x_{t + 1} = z_{t + 1}$;
		\ELSE
		\STATE $v_{t + 1} = \Px{\frac{1}{\tau}\breve{g}}{x_t - \frac{1}{\tau} \nabla \bar{f}(x_t)}$;
		\STATE $x_{t + 1} = 
		\begin{cases}
		z_{t + 1} & F(z_{t + 1}) \le F(v_{t + 1}) \\
		v_{t + 1} & \text{otherwise}
		\end{cases}$;
		\ENDIF
		\STATE $\alpha_{t + 1} = \frac{1}{2}(\sqrt{4 \alpha_t^2 + 1} + 1 )$;
		\STATE $q_{t + 1} = \eta q_t + 1$;
		\STATE $c_{t + 1} = \frac{\eta q_t c_t + F(x_{t + 1})}{q_{t + 1}}$;
		\ENDFOR
		\RETURN $x_{T + 1}$; 
	\end{algorithmic}
	\label{alg:nmAPG}
\end{algorithm}

As mentioned in Section~\ref{sec:shift}, the proposed decomposition of the nonconvex regularizer $g$ can
be regarded as a 
DC decomposition, 
which is not unique in general.
For example, we might try to add a quadratic term to convexify the nonconvex regularizer.
Specifically, we can 
decompose $g(x)$ in (\ref{eq:noncvx_gp})
as $\tilde{\varsigma}(x) + \hat{\varsigma}(x)$, where
\begin{equation}
\label{eq:spgupDCg1} 
\tilde{\varsigma}(x) 
= \lambda \sum_{i = 1}^d \left(\kappa(|x_i|) + \frac{\rho}{2} x_i^2 \right)
+ \sum_{j = 1}^{K}\mu_j \left( \kappa( \NM{x_{\mathcal{G}_j}}{2} ) + \frac{\rho}{2}
\NM{x_{\mathcal{G}_j}}{2}^2 \right),
\end{equation}
and
$\hat{\varsigma}(x) = - \frac{\rho}{2} \sum_{j = 1}^{K} (\mu_j + \lambda) \NM{x_{\mathcal{G}_j}}{2}^2$.
It can be easily shown that
$\hat{\varsigma}$ is concave, and
Proposition~\ref{pr:anotherDC} shows that
$\tilde{\varsigma}$ is convex.
Thus, $F$ can be transformed as $F(x) = \bar{f}(x) + \tilde{\varsigma}(x)$,
where $\bar{f}(x) = f(x) + \hat{\varsigma}(x)$ is Lipschitz-smooth, and $\tilde{\varsigma}$ is convex but nonsmooth.
However, the proximal step associated with $\tilde{\varsigma}$ has no simple closed-form solution.

\begin{proposition} 
\label{pr:anotherDC}
$\kappa(\NM{\cdot}{2}) + \frac{\rho}{2} \NM{\cdot}{2}^2$ is convex.
\end{proposition}


\subsubsection{Nonconvex Tree-Structured Group Lasso}
\label{sec:tree}

In (convex) tree-structured group lasso 
\citep{liu2010moreau,jenatton2011proximal},
the dimensions in $x$ are organized
as nodes in a tree, and each group corresponds to a subtree.
The regularizer 
is of the form
$\sum_{j = 1}^{K} \lambda_j \NM{x_{\mathcal{G}_j}}{2}$.
Interested readers are referred to \citep{liu2010moreau} for details.

For the nonconvex extension,
$g(x)$ becomes $\sum_{j = 1}^{K} \lambda_j \kappa (\NM{x_{\mathcal{G}_j}}{2} )$.
Again, there is no closed-form solution of its proximal step.
On the other hand,
the convexified regularizer is 
$\breve{g}(x) \equiv \kappa_0 \sum_{j = 1}^{K}\lambda_j \NM{x_{\mathcal{G}_j}}{2}$.
As shown in 
\citep{liu2010moreau},
its proximal step 
can be computed efficiently by processing all the groups once in some appropriate order.


\subsubsection{Nonconvex Total Variation (TV) Regularizer}
\label{sec:tv}

In an image,
nearby pixels 
are usually strongly correlated.
The TV regularizer captures such behavior by assuming that changes between nearby pixels are small.
Given an image $X \in \R^{m \times n}$, the TV regularizer is defined as
$\text{TV}(X) =  \NM{D_v X}{1} + \NM{X D_h}{1}$
\citep{nikolova2004variational},
$
D_v =
\begin{bmatrix}
-1 & 1     &       &  \\
& \ddots & \ddots &  \\
&       & -1    & 1
\end{bmatrix}
\in \R^{(m - 1) \times m}
$
and 
$
D_h =
\begin{bmatrix}
-1 &       &  \\
1       &  \ddots &  \\
& \ddots & - 1 \\
&       & 1
\end{bmatrix}
\in \R^{n \times (n - 1)}
$
are the horizontal and vertical partial derivative operators, respectively.
Thus, it is popular on image processing problems,
such as image denoising and deconvolution \citep{nikolova2004variational,beck2009fasttv}.

As in previous sections, the nonconvex extension of TV regularizer can be defined as
\begin{equation} \label{eq:noncvx_tv}
\sum_{i = 1}^{m - 1} \sum_{j = 1}^m \kappa\left( \left|\left[ D_v X \right]_{ij}\right|\right)
+ \sum_{i = 1}^{n} \sum_{j = 1}^{n - 1} \kappa\left( \left|\left[ X D_h \right]_{ij}\right|\right).
\end{equation}
Again, it is not clear how its proximal step can be efficiently computed.
However, with the proposed transformation,
the transformed problem is
\[ \min_{X} \bar{f}(X) + \mu \kappa_0 \TV{X}, \]
where $\mu$ is the regularization parameter,
$\bar{f}(X) =  f(X) + \mu \sum_{i = 1}^{m - 1} \sum_{j = 1}^m ( \kappa(
| [ D_v X ]_{ij}|) - \kappa_0 |[ D_v X ]_{ij}|) 
+ \mu \sum_{i = 1}^{n} \sum_{j = 1}^{n - 1} (  \kappa( |[ X
D_h]_{ij}|) - \kappa_0 |[ X D_h ]_{ij}| )$ is concave and Lipschitz smooth.
One then only needs to compute the proximal step of the standard TV regularizer.

However, unlike  the proximal steps in 
Sections~\ref{sec:gpslasso} and \ref{sec:tree}, the proximal step of the TV regularizer has no
closed-form solution and needs to be solved iteratively.
In this case,
\cite{schmidt2011convergence} showed that 
using inexact 
proximal steps can make proximal algorithms faster.
However, they only considered the situation where both $f$ and $g$ are convex.
In the following, we
extend nmAPG (Algorithm~\ref{alg:nmAPG}),
which can be used with
nonconvex objectives, to allow for inexact proximal steps
(steps~5 and 9 of Algorithm~\ref{alg:inexactAPG}).
However, Lemma~2 
of \citep{li2015accelerated}, which is key to the convergence of nmAPG, 
no longer holds dues to inexact proximal step.
To fix this problem,
in step~6 of Algorithm~\ref{alg:inexactAPG},
we use $F(X_t)$ 
instead  of $c_t$ in Algorithm~\ref{alg:nmAPG}.
Besides, we also drop the
comparison of $F(Z_{t + 1})$ and $F(V_{t + 1})$ (originally 
in step~9 of Algorithm~\ref{alg:nmAPG}).

\begin{algorithm}[ht]
\caption{Inexact nmAPG.}
\begin{algorithmic}[1]
	\STATE Initialize $\tilde{Z}_1 = X_1 = X_0$, $\alpha_0 = 0$, $\alpha_1 = 1$ and stepsize $\tau > \bar{L}$, $\delta \in (0, \tau - \bar{L})$; 
	\FOR{$t = 1, \dots, T$}
	\STATE choose tolerance $\epsilon_t$;
	\STATE $Y_t = X_t + \frac{\alpha_{t - 1}}{\alpha_t}(Z_t - X_t) + \frac{\alpha_{t - 1} - 1}{\alpha_t}(X_t - X_{t - 1})$;
	\STATE $\tilde{Z}_{t + 1}$ = approximate $\Px{\frac{1}{\tau}\breve{g}}{Y_t - \frac{1}{\tau} \nabla
	\bar{f}(Y_t)}$, with inexactness $\vartheta_{t + 1} \le \epsilon_t$;
	\IF{$F(\tilde{Z}_{t + 1}) \le F(X_t) - \frac{\delta}{2}\NM{\tilde{Z}_{t + 1} - Y_t}{F}^2$}
	\STATE $X_{t + 1} = \tilde{Z}_{t + 1}$;
	\ELSE
	\STATE $X_{t + 1}$ = approximate $\Px{\frac{1}{\tau}\breve{g}}{X_t - \frac{1}{\tau} \nabla
	\bar{f}(X_t)}$, with inexactness $\vartheta_{t + 1} \le \epsilon_t$;
	\ENDIF
	\STATE $\alpha_{t + 1} = \frac{1}{2}(\sqrt{4 \alpha_t^2 + 1} + 1 )$;
	\ENDFOR
	\RETURN $X_{T + 1}$; 
	\end{algorithmic}
	\label{alg:inexactAPG}
\end{algorithm}

Inexactness of the proximal step can be controlled as follows.
Let 
$P = X - \frac{1}{\tau} \nabla \bar{f}(X)$, and
$h(X) \equiv \frac{1}{2} \NM{X - P}{F}^2 + \frac{1}{\tau} \breve{g}(X)$
be the objective in 
$\Px{\frac{1}{\tau}\breve{g}}{P}$.
As
$\breve{g}(X) = \kappa_0 \TV{X}$ is convex, $h$ is also convex.
Let $\tilde{X}$ be an inexact solution of this proximal step.
The inexactness $h(\tilde{X}) - h(\Px{\frac{1}{\tau} \breve{g}}{P})$ is upper-bounded by the duality gap 
$\vartheta \equiv h(\tilde{X}) - \mathcal{D}(\tilde{W})$,
where $\mathcal{D}$ is the dual of $h$, and $\tilde{W}$ is the corresponding dual variable.
In step~5 
(resp.
step~9)
of Algorithm~\ref{alg:inexactAPG},
we solve the proximal step until its duality gap
$\vartheta_{t + 1}$
is smaller than a given threshold $\epsilon_t$.
The following Theorem shows convergence of Algorithm~\ref{alg:inexactAPG}.


\begin{theorem}
\label{thm:convAPG}
Let
$\sum_{t = 1}^{\infty} \epsilon_t < \infty$.
The sequence $\{ X_t \}$ generated from Algorithm~\ref{alg:inexactAPG}
has at least one limit point,
and every limit point is also a critical point of \eqref{eq:pro}.
\end{theorem}


If the proximal step is exact,
$\NM{V_t - \Px{\frac{1}{\tau}\breve{g}}{V_t - \frac{1}{\tau} \nabla \bar{f}(V_t)}}{F}^2$
can be used to measure how far $V_t$ is from a critical point \citep{gongZLHY2013,ghadimi2016accelerated}.
In Algorithm~\ref{alg:inexactAPG},
the proximal step
is inexact, and
$X_{t + 1} $
is an inexact solution to $\Px{\frac{1}{\tau}\breve{g}}{V_t - \frac{1}{\tau} \nabla
\bar{f}(V_t)}$, where
$V_t = Y_t$ if step~7 is executed, and $V_t = X_t$ 
if step~9 is executed.
As $X_{t + 1}$ 
converges to a critical point
of \eqref{eq:pro},
we propose using
$d_t \equiv \NM{X_{t + 1} - V_t}{F}^2$
to measure how far $X_{t+1}$ is from a critical point.
The following Proposition shows a $O(1/T)$ convergence rate on 
$\min_{t = 1, \dots, T}d_t$.

\begin{proposition} \label{pr:inexactrate}
(i) $\lim_{t \rightarrow \infty} d_t = 0$; 
and (ii) $\min_{t = 1, \dots, T} d_t$ converges to zero at a rate of $O(1/T)$.
\end{proposition}

Note that the (exact) nmAPG in Algorithm~\ref{alg:nmAPG} cannot handle the nonconvex $g$ in
\eqref{eq:noncvx_tv} efficiently, 
as the corresponding proximal step has no closed-form solutions but has to be solved exactly.
Even the proposed inexact nmAPG (Algorithm~\ref{alg:inexactAPG}) cannot be directly used with nonconvex $g$.
As the dual of the nonconvex proximal step is difficult to derive and the optimal duality gap is nonzero in general,
the proximal step's inexactness cannot be easily controlled.


\subsection{Frank-Wolfe Algorithm}
\label{sec:noncvxFW}


In this section,
we use the Frank-Wolfe algorithm to
learn
a low-rank matrix
$X\in \R^{m\times n}$ for matrix completion as reviewed in Section~\ref{sec:fw}.
The nuclear norm regularizer in 
(\ref{eq:lowrank})
may over-penalize top singular values.
Recently, there is growing interest to replace this with nonconvex regularizers 
\citep{lu2014generalized,lu2015generalized,qyao2015icdm,goqqlow2016}.
Hence, instead of (\ref{eq:lowrank}),
we consider
\begin{align} \label{eq:lowrank2}
\min_{X} f(X) + \mu \sum_{i = 1}^m \kappa(\sigma_i(X)).
\end{align}
When $\kappa$ is the identity function, 
(\ref{eq:lowrank2}) reduces to 
(\ref{eq:lowrank}).
Note that the FW algorithm cannot be directly used on 
(\ref{eq:lowrank}), as its linear subproblem in 
(\ref{eq:subprob})
then becomes
$\min_{S:\sum_{i = 1}^m \kappa(\sigma_i(S)) \le 1} \langle S, \nabla f(X_t) \rangle$,
which is difficult to osolve.  

Using Proposition~\ref{prop:c3}, problem \eqref{eq:lowrank2} is transformed into
\begin{align} \label{eq:equiv:lowrank}
\min_{X}
\bar{f}(X) + \bar{\mu} \NM{X}{*},
\end{align}
where 
\begin{equation} \label{eq:lowf}
\bar{f}(X) = f(X) + \bar{g}(X), \;\; \bar{g}(X) = \mu\sum_{i = 1}^m (\kappa(\sigma_i(X)) - \kappa_0 \sigma_i(X) ),
\end{equation} 
and $\bar{\mu}=\mu \kappa_0$.
This only involves the standard nuclear norm regularizer. However,
Algorithm~\ref{alg:stdfw} still cannot be used as
$\bar{f}$ in 
(\ref{eq:lowf})
is no longer convex.
A FW variant allowing nonconvex 
$\bar{f}$ 
is proposed in \citep{bredies2009generalized}. However, 
condition~1 in \citep{bredies2009generalized}
requires
$g$ to satisfy  
$\lim_{\NM{X}{F} \rightarrow \infty} \frac{g(X)}{\NM{X}{F}} = \infty$.
Such condition does not hold with $g(X) = \NM{X}{*}$
in (\ref{eq:equiv:lowrank}) as
\begin{align*}
\frac{\NM{X}{*}}{\NM{X}{F}}
= \sqrt{\frac{(\sum_{i = 1}^m \sigma_i)^2}{\sum_{i = 1}^m \sigma_i^2}}
\le \sqrt{\frac{m \sum_{i = 1}^m \sigma_i^2}{\sum_{i = 1}^m \sigma_i^2}}
= \sqrt{m} <\infty.
\end{align*}

In the following, we propose a nonconvex FW variant (Algorithm~\ref{alg:novelfw})
for the transformed problem \eqref{eq:equiv:lowrank}.
It is similar to Algorithm~\ref{alg:stdfw}, but with three important modifications.
First, 
$\bar{g}(X)$ in (\ref{eq:lowf}) depends on the singular values of $X$, which cannot be directly obtained
from the $UV^{\top}$ factorization in \eqref{eq:uv}.
Instead, we use  the low-rank factorization 
\begin{equation} \label{eq:fact}
X = U B V^{\top},
\end{equation} 
where $U \in \R^{m \times k}$, $V \in \R^{n \times k}$ are orthogonal 
and $B
\in \mathcal{S}_+^{k \times k}$ 
is 
positive 
semidefinite.


\begin{algorithm}[H]
\caption{Frank-Wolfe algorithm for solving the nonconvex problem \eqref{eq:equiv:lowrank}.}
\begin{algorithmic}[1]
\STATE $U_1 = [\;]$, $B_1 = [\;]$ and $V_1 = [\;]$;
\FOR{$t = 1 \dots T$}
\STATE $[u_t, s_t, v_t] = \text{rank1SVD}(\nabla \bar{f}(X_t))$;
\STATE obtain $\alpha_t$ and $\beta_t$ from \eqref{eq:qudra};
		\STATE $[\bar{U}_t, \bar{B}_t, \bar{V}_t] = \text{warmstart}(U_{t}, u_t, V_{t}, v_t, B_{t}, \alpha_t, \beta_t)$;
		\STATE obtain $[U_{t + 1}, B_{t + 1}, V_{t + 1}]$ from \eqref{eq:local}, using $\bar{U}_t$, $\bar{B}_t$ and
		$\bar{V}_t$ for warm-start; 
		\\ // $X_{t + 1} = U_{t + 1} B_{t + 1} V_{t + 1}^{\top}$
		\ENDFOR
		\RETURN $U_{T + 1}$, $B_{T + 1}$ and $V_{T + 1}$.
	\end{algorithmic}
	\label{alg:novelfw}
\end{algorithm}

The second problem is that line search  
in Algorithm~\ref{alg:stdfw} is
inefficient in general when operated on a nonconvex $\bar{f}$.
Specifically,
step~4
in Algorithm~\ref{alg:stdfw}  then becomes
\begin{align}
\label{eq:fwlinesea}
[\alpha_t, \beta_t]
= \arg\min_{\alpha \ge 0, \beta \ge 0} \bar{f}( \alpha X_t + \beta u_t v_t^{\top} )
+ \bar{\mu} (\alpha\NM{X_t}{*} + \beta).
\end{align}
To solve
(\ref{eq:fwlinesea}),
we have to compute
$\frac{\partial \bar{f}(S)}{\partial \alpha}$ and $\frac{\partial \bar{f}(S)}{\partial \beta}$,
where $S = \alpha X_t + \beta u_t v_t^{\top}$.
As shown in Proposition~\ref{pr:gradkappa},
this requires the SVD of $S$ and can be expensive.

\begin{proposition} \label{pr:gradkappa}
Let the SVD of 
$S$
be  $U_S \Diag{[\sigma_1(S), \dots ,\sigma_m(S)]} V_S^{\top}$.  Then
\begin{align*}
\frac{\partial \bar{f}(S)}{\partial \alpha} 
= \alpha \langle X_t, \nabla \bar{f}(S) \rangle, 
\quad \text{and} \quad
\frac{\partial \bar{f}(S)}{\partial \beta} 
= \beta u_t^{\top} \nabla \bar{f}(S) v_t,
\end{align*}
where $\nabla \bar{f}(S) = \nabla f(S) + \mu U_S \Diag{w} V_S^{\top}$,
and $w = [\kappa'(\sigma_i(S)) - \kappa_0]
\in \R^m$.
\end{proposition}

\begin{corollary} \label{pr:gradkappa2}
For $X$ in (\ref{eq:fact}),
let the SVD of $B$ be  $U_B \Diag{[\sigma_1(B), \dots, \sigma_k(B)]} V_B^{\top}$.
Then, $\nabla \bar{f}(X) = \nabla f(X) + \bar{\mu} (U U_B) \Diag{w} (V V_B)^{\top}$, 
where $w = [\kappa'(\sigma_i(B)) - \kappa_0]
\in \R^k$.
\end{corollary}
Alternatively, as $S$ is a rank one updates of $X_t$, one can perform incremental update on
SVD, which takes $O((m + n)t^2)$ time \citep{golub2012matrix}.  
However, every time $\alpha,\beta$ are changed, 
this incremental SVD has to be recomputed, and is thus inefficient.

To alleviate this problem, we approximate $\bar{f}(S)$ by the upper bound as
\begin{eqnarray}
\bar{f}(S) & = & 
\bar{f}(X_t + (\alpha - 1) X_t + \beta u_t v_t^{\top})
\notag \\
& \le & 
\bar{f}(X_t) + \langle (\alpha - 1) X_t + \beta u_t v_t^{\top}, \nabla \bar{f}(X_t) \rangle + \frac{\bar{L}}{2}\NM{(\alpha - 1) X_t + \beta u_t v_t^{\top}}{F}^2.
\label{eq:temp1}
\end{eqnarray}
As $(u_t, v_t)$ is obtained from the rank-$1$ SVD of $\nabla \bar{f}(X_t)$,
we have
$\NM{u_t v_t^{\top}}{F} = 1$ and 
$u_t^\top \nabla \bar{f}(X_t) v_t = s_t$.
Moreover,
$X_t = U_t B_t V_t^{\top}$,
and so $\NM{X_t}{F} = \NM{B_t}{F}$ and $\NM{X_t}{*} = \Tr{B_t}$.
Substituting these and the upper bound \eqref{eq:temp1} into 
	\eqref{eq:fwlinesea},
	we obtain
	a simple quadratic program:
\begin{eqnarray}
& \min_{\alpha\ge 0,\beta\ge 0} &
\frac{(\alpha - 1)^2\bar{L}}{2} \NM{B_t}{F}^2 + (\alpha - 1) \beta \bar{L} (u_t^{\top} U_t) B_t (V_t^{\top} v_t) 
+ \frac{\beta^2\bar{L}}{2} + \beta s_t
\nonumber\\
&&+ \alpha \langle  B_t , U_t^{\top} \nabla \bar{f}(X_t) V_t \rangle 
+ \bar{\mu}( \alpha \NM{B_t}{*} + \beta ).
\label{eq:qudra} 
\end{eqnarray}
Note that the objective in \eqref{eq:qudra} is convex,
as the RHS in \eqref{eq:temp1} is convex and the last term from \eqref{eq:fwlinesea} is affine.
Moreover, using Corollary~\ref{pr:gradkappa2}, $\langle  B_t , U_t^{\top} \nabla \bar{f}(X_t) V_t \rangle $ in
\eqref{eq:qudra} 
can be obtained as 
\begin{align*}
\langle  B_t, U_t^{\top} \nabla \bar{f}(X_t) V_t \rangle 
= \langle  B_t, U_t^{\top} \nabla f(X_t) V_t \rangle 
+ \bar{\mu} \sum_{i = 1}^t \sigma_i(B_t)( \kappa'(\sigma_i(B_t)) - \kappa_0 ).
\end{align*}
Instead of requiring SVD on $X_t$, it only requires
	SVD on $B_t$ (which is of size $t \times t$ at the $t$th iteration of
	Algorithm~\ref{alg:novelfw}). 
 As the target matrix is supposed to be
low-rank, $t \ll m$.  Hence, all the coefficients in \eqref{eq:qudra} can be obtained in $O((m + n)t^2 + \NM{\Omega}{1}t)$ time.
Besides, \eqref{eq:qudra} is a quadratic program with only two variables, and thus can be very efficiently
solved.


The third modification is that with $\bar{f}$ instead of $f$,
\eqref{eq:fwfactor} can no longer be used for local optimization,
as $\bar{g}(X)$ in (\ref{eq:lowf}) depends on the singular values of $X$.
On the other hand,
with the decomposition  
of $X$ in (\ref{eq:fact})
and Proposition~\ref{pr:svreduce} below,
\eqref{eq:equiv:lowrank} 
can be rewritten as
\begin{eqnarray}
& \min_{U, B, V} & f(U B V^{\top}) + \bar{g}(B) + \bar{\mu} \Tr{B}
\label{eq:local} \\
& \text{s.t.} & 
U^{\top}U = I, 
V^{\top}V = I,
B \in \mathcal{S}_+.
\label{eq:ortho}
\end{eqnarray}
This can be efficiently solved using matrix optimization techniques on the Grassmann
manifold \citep{ngo2012scaled}.

\begin{proposition} \label{pr:svreduce}
For orthogonal matrices $U$ and $V$, $\bar{g}( U B V^{\top} ) = \bar{g}(B)$.
\end{proposition}

In Algorithm~\ref{alg:novelfw},
step~5 
is used to warm-start
(\ref{eq:local}),
and the procedure is shown in Algorithm~\ref{alg:warmst}.
It expresses 
$X_t = \alpha_t U_{t - 1} B_{t - 1} V_{t - 1}^{\top} + \beta_t u_t v_t^{\top}$
obtained in step~4
to the form $U_t B_t V_t^{\top}$  
so that the orthogonal constraints on 
$U_t, V_t$  
in \eqref{eq:ortho} are satisfied.




\begin{algorithm}[ht]
	\caption{$\text{warmstart}(U_t, u_t, V_t, v_t, B_t, \alpha_t, \beta_t)$.}
	\begin{algorithmic}[1]
		\STATE $[\bar{U}_t, R_{\bar{U}_t}] = \text{QR}([U_t, u_t])$; // QR denotes the  QR factorization
		\STATE $[\bar{V}_t \,, R_{\bar{V}_t}] = \text{QR}([V_t, \, v_t])$;
		\STATE $\bar{B}_t = R_{\bar{U}_t}
		\begin{bmatrix}
		\alpha_t B_t & 0 \\
		0   & \beta_t
		\end{bmatrix}
		R_{\bar{V}_t}^{\top}$;
		\RETURN $\bar{U}_t$, $\bar{B}_t$ and $\bar{V}_t$; 
	\end{algorithmic}
	\label{alg:warmst}
\end{algorithm}

Existing analysis for the FW algorithm cannot be used on this nonconvex problem.
The following Theorem shows convergence of Algorithm~\ref{alg:novelfw} 
to a critical point of \eqref{eq:lowrank}.

\begin{theorem}
\label{thm:conv:ncg}
If \eqref{eq:lowrank} has a rank-$r$ critical point,
then 
Algorithm~\ref{alg:novelfw}
converges to a critical point of \eqref{eq:lowrank}
after $r$ iterations.
\end{theorem}


\subsection{Alternating Direction Method of Multipliers (ADMM) }
\label{sec:ADMM}

In this section,
we consider using ADMM on the consensus optimization problem
\eqref{eq:ADMMpro2}.
When all the $f_i$'s and $g$ are convex, ADMM has a convergence rate of $O(1/T)$ 
\citep{he20121}.
Recently, ADMM has been extended to problems where
$g$ is convex but $f_i$'s are nonconvex
\citep{hong2016convergence}.
However, 
when $g$ is nonconvex,
such as when a nonconvex regularizer is used in regularized risk minimization,
the convergence of ADMM is still an open reseach problem.

Using the proposed transformation, we can
decompose a nonconvex $g$ as  $\bar{g} + \breve{g}$, where
$\bar{g}$  is concave and Lipschitz-smooth, while $\breve{g}$ is convex but possibly nonsmooth.
Problem~\eqref{eq:ADMMpro2} can then be rewritten as
\begin{align}
\min_{y, x^1, \dots, x^M}
\sum_{i = 1}^M \bar{f}_i(x^i) + \breve{g}(y) \;\; :\; \; x^1 = \dots = x^M = y, 
\label{eq:ADMMpro3}
\end{align}
where $\bar{f}_i(x) = f_i(x) + \frac{1}{M} \bar{g}(x)$.
Let 
$p^i$ be the dual variable for the constraint $x^i = y$.
The augmented Lagrangian for \eqref{eq:ADMMpro3} is
\begin{align}
\mathcal{L}\left( y, x^1, \dots, x^M , p^1, \dots, p^M \right) 
= \breve{g}(y) + \sum_{i = 1}^M \bar{f}_i(x^i)
+ (p^i)^{\top}(x^i - y) + \frac{\tau}{2}\NM{x^i - y}{2}^2.
\label{eq:augLang}
\end{align}
Using
\eqref{eq:ADMM1} and
\eqref{eq:ADMM2},
we have the following update equations at iteration $t$:
\begin{eqnarray}
x^i_{t + 1}
& = & \arg\min_{x^i} \bar{f}_i(x^i) + (p^i_t)^{\top}(x^i - y_t) + \frac{\tau}{2}\NM{x^i - y_t}{2}^2, \nonumber \\
y_{t + 1} 
& = & \arg\min_{y}
\frac{1}{2}\left\|y - \sum_{i = 1}^M \left(x^i_t + \frac{1}{\tau} p^i_t \right) \right\|_2^2 + \frac{1}{\tau} \breve{g}(y)
= \text{prox}_{\frac{1}{\tau} \breve{g}} \left(\sum_{i = 1}^M x^i_t + \frac{1}{\tau} p^i_t
\right). \label{eq:tmp2}
\end{eqnarray}
As in previous sections,
the proximal step in (\ref{eq:tmp2}), which is associated with the convex $\breve{g}$,
is usually easier to compute than the proximal step associated with the original nonconvex $g$.
Moreover, since $\breve{g}$ is convex, convergence results in
Theorem 2.4 of \citep{hong2016convergence} can now be applied.
Specifically,
the sequence 
$\{ y_t, \{x^i_t\} \}$ 
generated by the ADMM procedure
converges to a critical point of \eqref{eq:ADMMpro3}.


\subsection{Stochastic Variance Reduced Gradient}
\label{sec:svrg}

Variance reduction methods 
have been commonly used to speed up the often slow convergence of stochastic gradient descent (SGD).
Examples are
stochastic variance reduced gradient (SVRG) \citep{johnson2013accelerating}
and its proximal extension Prox-SVRG \citep{xiao2014proximal}.
They can be used for the following optimization problem
\begin{align}
\min_x \sum_{i = 1}^N \ell(y_i, a_i^{\top} x) + g(x),
\label{eq:stocha}
\end{align}
where 
$\{(a_1,y_1),\dots, (a_N,y_N) \}$ are the training samples,
$\ell$ is a smooth convex loss function,
and $g$ is a convex regularizer.
Recently, Prox-SVRG is also extended for nonconvex objectives.
\cite{reddiHSPS2016} and \cite{zhuH2016} considered
smooth nonconvex $\ell$ but without $g$. This is further extended 
to the case of smooth $\ell$ and convex nonsmooth $g$
in \citep{reddi2016fast}. 
However, 
convergence is still unknown for the more general case where the regularizer $g$ is also nonconvex.

Using the proposed transformation,
\eqref{eq:stocha} can be rewritten as
\begin{align*}
\min_{x} 
\sum_{i = 1}^N \left( \ell(y_i, a_i^{\top} x) + \frac{1}{N} \bar{g}(x) \right) + \breve{g}(x),
\end{align*}
where $\ell + \frac{1}{N} \bar{g}$
is smooth and $\breve{g}$ is convex.
As a result, convergence results in \citep{reddi2016fast} can now be applied.



\subsection{With OWL-QN}
\label{sec:quasiNT}

In this section, we consider 
OWL-QN \citep{andrew2007scalable} and its variant mOWL-QN \citep{gong2015modified},
which are
efficient algorithms
for the $\ell_1$-regularization problem
\begin{equation} \label{eq:owl}
\min_x f(x) + \mu \NM{x}{1}.
\end{equation} 
Recently,
\citeauthor{gong2015honor}
(\citeyear{gong2015honor}) 
proposed 
a nonconvex generalization 
for (\ref{eq:owl}),
in which the standard $\ell_1$ regularizer is replaced by the nonconvex 
$g(x)= \mu \sum_{i = 1}^d \kappa(|x_i|)$:
\begin{align}
\min_x 
f(x) + \mu \sum_{i = 1}^d \kappa(|x_i|).
\label{eq:lasso}
\end{align}
\citeauthor{gong2015honor}
(\citeyear{gong2015honor}) proposed a sophisticated 
algorithm (HONOR) which involves a combination of 
quasi-Newton and gradient descent steps.
Though the algorithm is similar to 
OWL-QN 
and mOWL-QN,
the convergence analysis in  \citep{gong2015modified} cannot be directly applied
as
the regularizer is nonconvex.
Instead, a non-trivial extension was developed in 
\citep{gong2015honor}.

Here, by convexifying the nonconvex regularizer, 
\eqref{eq:lasso} can be rewritten as
\begin{equation} \label{eq:lasso:equiv}
\min_x \bar{f}(x) + \mu \kappa_0 \NM{x}{1},
\end{equation}
where
$\bar{f}(x) = f(x) + \bar{g}(x)$, 
and $\bar{g}(x) = \mu \sum_{i = 1}^d (\kappa(|x_i|) - \kappa_0 |x_i|)$.
It is easy to see that the analysis in \citep{gong2015modified} can be extended to handle smooth but nonconvex $\bar{f}$.
Thus, mOWL-QN is still guaranteed to converge to a critical point.

As demonstrated in previous sections, other DC decompositions of $g$  are not as useful.
For example, with the one in Proposition~\ref{pr:anotherDC},
we obtain 
the convex regularizer
$\breve{\varsigma}(x) = \frac{\rho \mu}{2}\NM{x}{2}^2 + \mu \sum_{i = 1}^d
\kappa(|x_i|)$. However,
mOWL-QN can no longer be applied, as it works only with the
$\ell_1$-regularizer.

Problem~(\ref{eq:lasso}) can be solved by either (i) directly using HONOR, or  (ii)
using mOWL-QN on the transformed problem \eqref{eq:lasso:equiv}.
We believe that the latter approach is computationally more efficient.
In \eqref{eq:lasso},
the Hessian depends on  both terms in the objective, as the 
second-order derivative 
of $\kappa$ is 
not zero in general.
However, HONOR 
constructs the approximate Hessian 
using only information from $f$, and thus ignores the curvature information due to $\sum_{i = 1}^d \kappa(|x_i|)$.
On the other hand,
the Hessian 
in (\ref{eq:lasso:equiv})
depends only on $\bar{f}$, as the Hessian due to 
$\NM{x}{1}$ is zero
\citep{andrew2007scalable},
and mOWL-QN now extracts Hessian from $\bar{f}$.
Hence, optimizing \eqref{eq:lasso:equiv} with mOWL-QN is potentially faster,
as all the second-order information is utilized.
This will be verified empirically in Section~\ref{sec:honor-expt}.


\subsection{Nonsmooth and Nonconvex Loss}
\label{sec:nonsthmloss}

In many applications,
besides having nonconvex regularizers,
the loss function may also be nonconvex and nonsmooth.
Thus, neither $f$ nor $g$ in \eqref{eq:pro} is convex, smooth.
The optimization problem becomes even harder, and
many existing algorithms cannot be used.
In particular, 
the proximal algorithm requires $f$ in \eqref{eq:pro} to be 
smooth (possibly nonconvex) \citep{gongZLHY2013,li2015accelerated,boct2016inertial}.
The FW algorithm requires $f$ in \eqref{eq:FW} to be smooth and convex
\citep{jaggi2013revisiting}.
For the ADMM, it allows $f$ 
in the consensus problem
to be smooth, 
but 
$g$ 
has to be convex
\citep{hong2016convergence}.
For problems of the form $\min_{x,z} f(y) + g(y):y = A x$,
ADMM requires
$A$ to have full row-rank
\citep{li2015global}.
As will be seen,
it is not satisfied  for problems considered in this 
section.
CCCP \citep{yuille2002concave} and smoothing
\citep{chen2012smoothing}
are more general and can still be used,
but are usually very slow.

In this section,
we consider two application examples,
and show how they can be efficiently solved with the proposed transformation.


\subsubsection{Total Variation Image Denoising}
\label{sec:tvdenoimg}

Using the $\ell_1$ loss and TV regularizer introduced in Section~\ref{sec:tv}, consider the following optimization problem:
\begin{align}
\min_{X} \NM{Y - X}{1} + \mu \TV{X},
\label{eq:imgTV}
\end{align}
where $Y \in \R^{m \times n}$ is a given corrupted image, and $X$ is the target image to
be recovered.
The use of nonconvex loss and regularizer
often produce better performance 
\citep{yan2013restoration}.
Thus, we consider the following nonconvex extension:
\begin{align}
\min_{X} \sum_{i = 1}^m \sum_{j = 1}^n \kappa\left( \left| \left[Y - X\right]_{ij} \right| \right) 
+ \mu \sum_{i = 1}^{m - 1} \sum_{j = 1}^m \kappa\left( \left|\left[ D_v X \right]_{ij}\right|\right)
+ \mu \sum_{i = 1}^{n} \sum_{j = 1}^{n - 1} \kappa\left( \left|\left[ X D_h \right]_{ij}\right|\right),
\label{eq:ncvximgTV}
\end{align}
where both the loss and regularizer are nonconvex and nonsmooth.
As discussed above, this can be solved by 
CCCP and smoothing.
However, as will be experimentally demonstrated in Section~\ref{sec:imgden}, their convergence
is slow.

Using the proposed transformation on both the loss and regularizer, problem~(\ref{eq:ncvximgTV}) can be transformed to the following problem:
\begin{align}
\min_{X} \bar{f}(X) + \kappa_0 \NM{X - Y}{1} + \kappa_0 \mu \TV{X},
\label{eq:imgTVtrans}
\end{align}
where 
\begin{align*}
\bar{f}(X) 
= & \sum_{i = 1}^m \sum_{j = 1}^n \kappa\left( \left| \left[Y - X\right]_{ij}\right| \right)
- \kappa_0 \NM{Y - X}{1} 
\\ 
& + \mu \left[ \sum_{i = 1}^{m - 1} \sum_{j = 1}^m \kappa\left( \left|\left[ D_v X \right]_{ij}\right|\right)
- \kappa_0 \NM{D_v X}{1}
+ \sum_{i = 1}^{n} \sum_{j = 1}^{n - 1} \kappa\left( \left|\left[ X D_h \right]_{ij}\right|\right) - \kappa_0 \NM{X D_h}{1} \right]
\end{align*}
is smooth and nonconvex. 
As \eqref{eq:imgTVtrans} is not a consensus problem, the method in
\citep{hong2016convergence} cannot be used.
To use the ADMM algorithm in \citep{li2015global},
extra variables and constraints $Z_v = D_v X$ and $Z_h = X D_h$ have to be imposed.
However, the
	full row-rank condition in 
\citep{li2015global} does not hold.

In this section, we consider the proximal algorithm.
Given some  $Z$,
the 
proximal step 
in (\ref{eq:imgTVtrans})
is
\begin{align}
\arg\min_{X}\frac{1}{2} \NM{X - Z}{F}^2 + \frac{1}{\tau} \left( \NM{X - Y}{1} + \mu \TV{X} \right),
\label{eq:tvprox}
\end{align}
where $\tau$ is the stepsize.
Though 
this has no closed-form solution,  
$\NM{X - Y}{1} + \mu \TV{X}$
in (\ref{eq:tvprox}) is convex and one can thus monitor
	inexactness of the proximal step 
	via the
	duality gap.
Thus,
we can use the proposed inexact nmAPG algorithm in Algorithm~\ref{alg:inexactAPG} for \eqref{eq:imgTVtrans}.
It can be shown that the dual of \eqref{eq:tvprox} is
\begin{eqnarray}
& \min_{W, P, Q}  & 
\frac{1}{2 \tau}\NM{W + \mu D_v^{\top} P + \mu Q D_h^{\top}}{F}^2
- \la Z, W  \ra
- \mu \la D_v Z, P \ra
- \mu \la Z D_h, Q \ra
+ \la Y, W \ra
\notag \\
& \st & 
\NM{W}{\infty} \le 1, 
\NM{P}{\infty} \le 1
\;\text{and}\;
\NM{Q}{\infty} \le 1,
\label{eq:tvproxdual}
\end{eqnarray}
and the primal variable can be recovered as $X = Z - \frac{1}{\tau}(W + \mu D_v^{\top}P + \mu Q D_h^{\top})$.
By substituting the obtained $X$ into \eqref{eq:tvprox} and $\{ W, P, Q \}$ into \eqref{eq:tvproxdual},
the duality gap can be computed in
$O(mn)$ time.
As \eqref{eq:tvproxdual} is a smooth and convex problem, 
both accelerated gradient descent \citep{nesterov2013gradient} and L-BFGS \citep{nocedal2006numerical} can be applied.
Algorithm~\ref{alg:inexactAPG} is then guaranteed to converge to a critical point of \eqref{eq:ncvximgTV}
(Theorem~\ref{thm:convAPG} and Proposition~\ref{pr:inexactrate}).

Note that it is more advantageous to transform both the loss and regularizer in (\ref{eq:imgTVtrans}).
If only the regularizer  in
(\ref{eq:ncvximgTV})
is transformed,
we obtain
\begin{align}
\bar{f}_{\text{TV}}(X)
+ \sum_{i = 1}^m \sum_{j = 1}^n \kappa\left( \left| \left[Y - X\right]_{ij} \right| \right) 
+ \kappa_0 \mu \TV{X},
\label{eq:trantx1}
\end{align}
where
\begin{align*}
\bar{f}_{\text{TV}}(X) = \mu \left[ \sum_{i = 1}^{m - 1} \sum_{j = 1}^m \kappa\left( \left|\left[ D_v X \right]_{ij}\right|\right)
- \kappa_0 \NM{D_v X}{1}
+ \sum_{i = 1}^{n} \sum_{j = 1}^{n - 1} \kappa\left( \left|\left[ X D_h \right]_{ij}\right|\right) - \kappa_0 \NM{X D_h}{1} \right]
\end{align*}
is nonconvex.
The corresponding proximal step for \eqref{eq:trantx1} is
\begin{align}
\arg\min_X \frac{1}{2}\NM{X - Z}{F}^2 
+ \frac{1}{\tau}
\left( \sum_{i = 1}^m \sum_{j = 1}^n \kappa\left( \left| \left[Y - X\right]_{ij} \right| \right)  + \kappa_0 \mu \TV{X} \right). 
\label{eq:tran1prox}
\end{align}
While the proximal steps in both \eqref{eq:tvprox} and \eqref{eq:tran1prox} have no
closed-form solution,
working with \eqref{eq:tvprox} is more efficient. As 
\eqref{eq:tvprox} 
is convex, its dual 
can be efficiently solved with methods such as accelerated gradient descent and L-BFGS.
In contrast, \eqref{eq:tran1prox} is nonconvex, its duality gap is nonzero, and so can
only be solved in the primal with slower methods like CCCP
and smoothing.
Besides, 
one can only use
the more expensive nmAPG (Algorithm~\ref{alg:nmAPG})
but not the proposed inexact proximal algorithm.

\subsubsection{Robust Sparse Coding}
\label{sec:robostsc}

The second application is robust sparse coding,
which has been popularly used in 
face recognition \citep{yang2011robust},
image analysis \citep{lu2013online} and background modeling \citep{zhao2011background}.
Given an observed signal $y \in \R^m$, the goal is to seek a robust sparse representation
$x \in \R^d$ 
of $y$ based on the dictionary $D \in \R^{m \times d}$
(which is assumed to be fixed here).
Mathematically, it is formulated as the following optimization problem:
\[ \min_{x}
\NM{y - D x}{1} + \mu \NM{x}{1}. \]

Its nonconvex extension is:
\begin{align}
\min_{ x } 
\sum_{j = 1}^m \kappa( |[y - D x]_j| )  + \mu \sum_{i = 1}^d \kappa( |x_i| ).
\label{eq:ncvxrbstdic}
\end{align}
Using the proposed transformation, problem \eqref{eq:ncvxrbstdic} becomes
\begin{align}
\min_{ x } \bar{f}(x) + \kappa_0 \NM{y - D x}{1} 
+ \mu \kappa_0 \NM{x}{1},
\label{eq:ncvxrbstdic2}
\end{align}
where 
\begin{align*}
\bar{f}(x) = 
\mu \sum_{j = 1}^d \kappa( |x_j| ) -  \kappa_0 \mu \NM{x}{1}  
+ \sum_{j = 1}^m \kappa( |[y - D x]_j| )
- \kappa_0 \NM{y - D x}{1}
\end{align*}
is smooth and nonconvex.
Again, we use the inexact nmAPG algorithm in Algorithm~\ref{alg:inexactAPG}.
The proximal step for \eqref{eq:ncvxrbstdic2} is 
\begin{align}
\arg\min_{ x }
\frac{1}{2}\NM{x - z}{2}^2 + \frac{1}{\tau} ( \NM{y - D x}{1} + \mu \NM{x}{1} ),
\label{eq:proxdl}
\end{align}
where $\tau$ is the stepsize and
$z$ is
given.
As in Section~\ref{sec:tvdenoimg},
$\NM{y - D x}{1} + \mu \NM{x}{1} $ 
in \eqref{eq:proxdl} is convex, and one can monitor inexactness of the proximal step 
by the duality gap.
The dual of \eqref{eq:proxdl} is
\begin{align}
\min_{p, q} \frac{1}{2 \tau}\NM{D^{\top} p + \mu q}{2}^2 - p^{\top} D z - \mu q^{\top} z
\;:\;
\NM{p}{\infty} \le 1,
\NM{q}{\infty} \le 1.
\label{eq:rscdual}
\end{align}
As in \eqref{eq:tvproxdual},
this can be solved with L-BFGS 
or accelerated gradient descent.
The primal variable can be recovered as $x = z - \frac{1}{\tau} (D^{\top} p + \mu q)$,
and the duality gap can be checked in $O(m d)$ time.

If only the regularizer is transformed, we obtain
\begin{align}
\min_{ x } 
\sum_{j = 1}^m \kappa( |[y - D x]_j| ) + \bar{f}_{\text{RSC}}(x) + \kappa_0 \mu \NM{x}{1},
\label{eq:rsc}
\end{align}
where 
$\bar{f}_{\text{RSC}}(x) = 
\mu \sum_{j = 1}^d \kappa( |x_j| ) -  \kappa_0 \mu \NM{x}{1}$.
The corresponding proximal step 
is
\begin{align}
\arg\min_x \frac{1}{2} \NM{x - z}{2}^2
+ \sum_{j = 1}^m \kappa( |[y - D x]_j| ) 
+  \kappa_0 \mu \NM{x}{1},
\label{eq:rscprox}
\end{align}
which 
still involve
the nonconvex function $\kappa$.
As in 
Section~\ref{sec:tvdenoimg},
\eqref{eq:rscdual} 
is easier to solve 
than \eqref{eq:rscprox}.


\section{Experiments}
\label{sec:expt}

In this section, we perform experiments on using the proposed procedure
with (i) proximal algorithms (Sections~\ref{sec:expt1}
and \ref{sec:tree-expt}); (ii) Frank-Wolfe algorithm (Section~\ref{sec:fw-expt}); 
(iii) comparision with HONOR
(Section~\ref{sec:honor-expt})
and (vi) image denoising (Section~\ref{sec:imgden}).
Experiments are performed on a PC with Intel i7 CPU and 32GB memory.
All algorithms are implemented in Matlab.


\subsection{Nonconvex Sparse Group Lasso}
\label{sec:expt1}

In this section, we perform experiments on 
the nonconvex sparse group lasso model in Section~\ref{sec:gpslasso}.
For simplicity, assume that
$\mu_1 = \dots = \mu_K = \mu$.
Using the square loss,
\eqref{eq:sglasso} becomes 
\begin{equation} \label{eq:expt1}
\min_{x} \frac{1}{2} \NM{y - A^{\top} x}{2}^2 
\!+\! \lambda \sum_{i = 1}^d \kappa(|x_i|) \!+\! \mu \sum_{j = 1}^{K} \kappa (\NM{x_{\mathcal{G}_j}}{2}),
\end{equation} 
where
$A = [a_1, \dots, a_N]$.
In this experiment, we use 
the LSP regularizer in Table~\ref{tab:regdecomp} (with $\theta = 0.5$)
as $\kappa(\cdot)$.
The synthetic data set is generated as follows. Let $d=10000$.
The ground-truth parameter $\bar{x} \in \R^{10000}$ is divided into $100$ non-overlapping groups: 
$\{1, \dots, 100\}$, $\{101,
\dots, 200\}$, $\dots$, $\{9901,\dots,10000\}$
(Figure~\ref{fig:synx}).
We randomly set $75\%$ of the
groups
to zero.
In each nonzero group, we randomly set $25\%$ of its features to zero,  and
generate
the nonzero features from the standard normal distribution
$\mathcal{N}(0, 1)$.
The whole data set has
$20,000$ samples, and
entries of the input matrix $A 
\in \R^{10000 \times 20000}$ 
are generated
from 
$\mathcal{N}(0, 1)$.
The ground-truth output is $\bar{y} = A^{\top} \bar{x}$. This is then corrupted by
random 
Gaussian noise
$\epsilon$
in $\mathcal{N}(0, 0.05)$ to produce
$y = \bar{y} + \epsilon$.

\begin{figure}[ht]
\centering
\includegraphics[width = 0.45\textwidth]{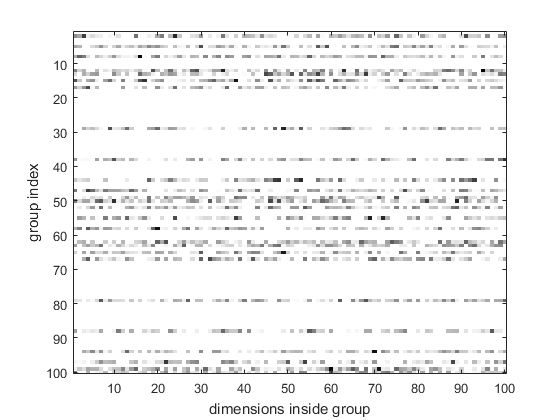}
\caption{An example ground-truth parameter $\bar{x} \in \R^{10000}$.
	It is reshaped as a 
	$100 \times 100$ 
	matrix, with each
	row representing
	a group.}
\label{fig:synx}
\end{figure}

The proposed algorithm will be called N2C (Nonconvex-to-Convex).
The proximal step of the
convexified regularizer 
$\breve{g}(x) = \kappa_0 ( \lambda \NM{x}{1} + \sum_{j = 1}^{K} \mu_j \NM{x_{\mathcal{G}_j}}{2} )$
is
obtained using the algorithm in \citep{yuan2011efficient}.
The nmAPG algorithm 
(Algorithm~\ref{alg:nmAPG})
in \citep{li2015accelerated} 
is used for optimization.
This will be compared with the following state-of-the-art algorithms:
\begin{enumerate}[itemsep = 0cm, topsep=0.125cm]
\item 
SCP: 
Sequential convex programming 
\citep{zhaosong2012}, in which the LSP regularizer is decomposed following \eqref{eq:spgupDCg1}.  

\item GIST \citep{gongZLHY2013}: Since the nonconvex regularizer is not separable, 
the associated proximal operator
has no closed-form solution.
Instead, we use SCP 
(with warm-start)
to solve it numerically.

\item 
GD-PAN
\citep{zhongK2014gdpan}:
It performs gradient descent with proximal average \citep{bauschke2008proximal} 
of the nonconvex regularizers.
Closed-form  solutions for the proximal operator of each regularizer are obtained 
separately, and then averaged.
	
\item nmAPG with the original nonconvex regularizer:
As in GIST, the proximal step is solved numerically by
SCP.

\item As a baseline, we also compare with the FISTA \citep{beck2009fast} algorithm,
which solves 
the convex sparse group lasso model (with 
$\kappa$
removed
from \eqref{eq:expt1}).
\end{enumerate}
We do not compare
with the concave-convex procedure \citep{yuille2002concave},
which has been shown to be slow \citep{gongZLHY2013,zhongK2014gdpan}.

We use
$50 \%$ of the data for training, another $25 \%$ as validation set to 
tune $\lambda, \mu$ in (\ref{eq:expt1}),
and the rest for testing.
The stepsize 
is fixed  at
$\tau = \sigma_{1}(A^{\top} A)$.
For performance evaluation, we use the
(i) testing root-mean-squared error (RMSE) on the predictions;
(ii) absolute error between the obtained parameter 
$\hat{x}$ with ground-truth $\bar{x}$:
$\text{ABS} = \NM{\hat{x} - \bar{x}}{1}/d$; and (iii) CPU time.
To reduce statistical variability, the experimental results are averaged over 5
repetitions.

Results are shown in Table~\ref{tab:syn:recover}.
As can be seen, all the nonconvex models obtain better errors (RMSE and ABS) than the convex
FISTA. As for the training speed, 
N2C is the fastest.
SCP, GIST, nmAPG and N2C targets the original problem \eqref{eq:pro},
and they have the same recovery performance.
GD-PAN solves an approximate problem in each of its iterations, and its error is slightly worse than the other nonconvex algorithms on this data set.

\begin{table}[ht]
	\centering
	\caption{Results on nonconvex sparse group lasso.  RMSE and ABS are scaled by $10^{-3}$, and the CPU time is in seconds.  The best and comparable results (according to the pairwise t-test with 95\% confidence)
		are highlighted.}
	\label{tab:syn:recover}
	\begin{tabular}{c | c | c  | c | c | c || c }
		\hline
		&             \multicolumn{3}{c|}{non-accelerated}              &       \multicolumn{2}{c||}{accelerated}        & convex        \\
		& SCP                   & GIST                  & GD-PAN        & nmAPG                 & N2C                    & FISTA         \\ \hline
		RMSE      & \textbf{50.6$\pm$2.0} & \textbf{50.6$\pm$2.0} & 52.3$\pm$2.0  & \textbf{50.6$\pm$2.0} & \textbf{50.6$\pm$2.0}  & 53.8$\pm$1.7  \\ \hline
		ABS      & \textbf{5.7$\pm$0.2}  & \textbf{5.7$\pm$0.2}  & 7.1$\pm$0.4   & \textbf{5.7$\pm$0.2}  & \textbf{5.7$\pm$0.2}   & 10.6$\pm$0.3  \\ \hline
		CPU time(sec)  & 0.84$\pm$0.14         & 0.92$\pm$0.12         & 0.94$\pm$0.22 & 0.65$\pm$0.06         & \textbf{0.48$\pm$0.05} & 0.79$\pm$0.14 \\ \hline
	\end{tabular}
\end{table}

Figure~\ref{fig:syn:comp} 
shows convergence of the objective with time and iterations 
for a typical run. 
SCP, GIST, nmAPG and N2C all converge towards the same objective value.
GD-PAN can only approximate the original problem.
Thus, it converges to an objective value which is larger than others.
nmAPG and N2C are based on the state-of-the-art proximal algorithm (Algorithm~\ref{alg:nmAPG}.
Both require nearly the same number of iterations for convergence (Figure~\ref{fig:syn:comp1}).
However, as N2C has cheap closed-form solution for its proximal step,
it is much faster when measured in terms of time (Figure~\ref{fig:syn:comp2}).
Overall,
N2C,
which uses acceleration and inexpensive proximal step,
is the fastest.

\begin{figure}[ht]
\centering
\subfigure[objective vs CPU time.]
{
\includegraphics[width = 0.45\textwidth]{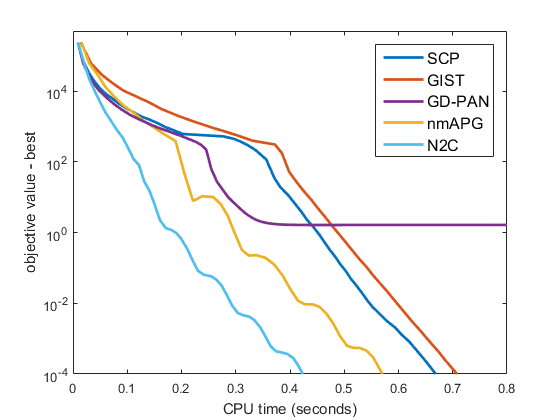}
\label{fig:syn:comp1}
}
\subfigure[objective vs iterations.]
{
\includegraphics[width = 0.45\textwidth]{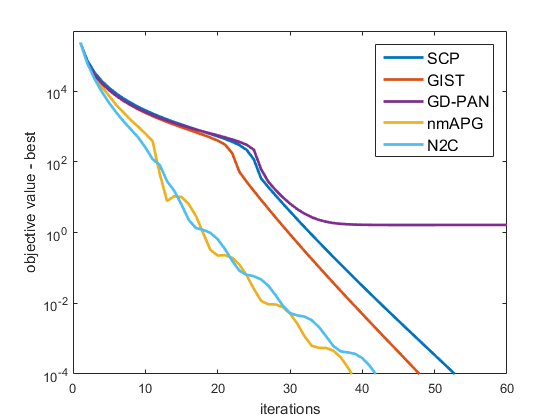}
\label{fig:syn:comp2}
}
\caption{Convergence of algorithms on nonconvex sparse group lasso. 
		FISTA is not shown as its (convex) objective is different from the others.}
\label{fig:syn:comp}
\end{figure}


\subsection{Nonconvex Tree-Structured Group Lasso}
\label{sec:tree-expt}

In this section,
we perform experiments on the nonconvex tree-structured group lasso model
in Section~\ref{sec:tree}.
We use the face data set \textit{JAFFE}\footnote{\url{http://www.kasrl.org/jaffe.html}}, which
contains $213$ images with seven facial expressions: 
anger, disgust, fear, happy, neutral, sadness and surprise.
Following \citep{liu2010moreau}, 
we resize each 
image 
from
$256 \times 256$ 
to $64 \times 64$. 
Their tree structure,
which is based on pixel neighborhoods,
is also used here.
The total 
number of groups 
$K$ is $85$.  

Since our goal is only to demonstrate usefulness of the proposed convexification scheme,  
we focus on the binary classification problem ``anger vs not-anger'' 
(with 30 anger images and 183 non-anger images).
The logistic loss is used, which is more appropriate
for classification.
Given training samples $\{(a_1,y_1),\dots, (a_N,y_N)\}$, 
the optimization problem is then 
\begin{align*}
\min_x\sum_{i = 1}^N \! w_i \log(1 + \exp( - y_i \cdot a_i^{\top} x ) ) 
+ \mu \sum_{i = 1}^{K} \lambda_i \kappa ( \NM{x_{\mathcal{G}_i}}{2} ),
\end{align*}
where 
$\kappa(\cdot)$ is the LSP regularizer (with $\theta=0.5$),, 
$w_i$'s are weights 
(set to be the reciprocal of the size of sample $i$'s class) used to alleviate class
imbalance, and 
$\lambda_i = 1 / \sqrt{\NM{\mathcal{G}_i}{1}}$ as in \citep{liu2010moreau}.
We use 
$60\%$ of the data 
for training, $20\%$ for validation 
and the rest for testing.
For the proposed N2C algorithm,
the proximal step of the convexified regularizer is obtained as in \citep{liu2010moreau}.

As in Section~\ref{sec:expt1}, 
it is compared with SCP, GIST, GD-PAN, nmAPG, and 
FISTA.
The stepsize $\eta$ is obtained by line search.
For performance evaluation, we use (i) the testing accuracy;
(ii) solution sparsity (i.e., percentage of nonzero elements);
and (iii) CPU time.
To reduce statistical variability, the experimental results are averaged over 5 repetitions.

Results are shown in Table~\ref{tab:face:recover}.
As can be seen,
all nonconvex models have similar testing accuracies, and they again outperform the
convex model. Moreover,
solutions
from the nonconvex models are sparser.
Overall, N2C is the fastest and has the
sparsest solution.

\begin{table}[ht]
	\centering
	\caption{Results on tree-structured group lasso.
		The best and comparable results (according to the pairwise t-test with 95\% confidence) are highlighted.}
	\begin{tabular}{c | c | c | c | c | c || c }
		\hline
		                      &                 \multicolumn{3}{c|}{non-accelerated}                  &       \multicolumn{2}{c||}{accelerated}       & convex       \\
		                      & SCP                   & GIST                  & GD-PAN                & nmAPG                 & N2C                   & FISTA        \\ \hline
		testing accuracy (\%) & \textbf{99.6$\pm$0.9} & \textbf{99.6$\pm$0.9} & \textbf{99.6$\pm$0.9} & \textbf{99.6$\pm$0.9} & \textbf{99.6$\pm$0.9} & 97.2$\pm$1.8 \\ \hline
		    sparsity (\%)     & 5.5$\pm$0.4           & 5.7$\pm$0.4           & 6.9$\pm$0.4           & 5.4$\pm$0.3           & \textbf{5.1$\pm$0.2}  & 9.2$\pm$0.2  \\ \hline
		    CPU time(sec)     & 7.1$\pm$1.6           & 50.0$\pm$8.1          & 14.2$\pm$2.6          & 3.8$\pm$0.4           & \textbf{1.9$\pm$0.3}  & 1.0$\pm$0.4  \\ \hline
	\end{tabular}
	\label{tab:face:recover}
\end{table}

Figure~\ref{fig:face:comp} shows 
convergence of the algorithms versus CPU time and number of iterations.
As can be seen, N2C is the fastest. 
GIST is the slowest, as it does not utilize acceleration and its proximal step is solved numerically which is expensive.
GD-PAN converges to a less optimal solution due to its use of approximation.
Moreover,
as in Section~\ref{sec:expt1},
nmAPG and N2C 
show similar convergence behavior w.r.t. the number of iterations (Figure~\ref{fig:face:comp2}),
but 
N2C is much faster w.r.t. time (Figure~\ref{fig:face:comp1}).

\begin{figure}[ht]
\centering
\subfigure[objective vs CPU time (seconds).]
{\includegraphics[width = 0.45\columnwidth]{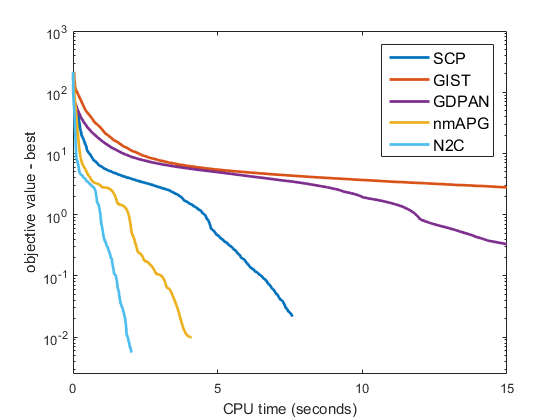}
\label{fig:face:comp1}}
\subfigure[the objective vs iterations.]
{\includegraphics[width = 0.45\columnwidth]{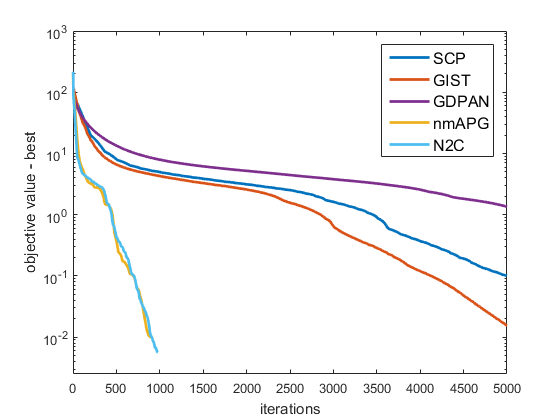}
\label{fig:face:comp2}}
\caption{Convergence of algorithms on nonconvex tree-structured group lasso.}
\label{fig:face:comp}
\end{figure}


\subsection{Nonconvex Low-Rank Matrix Completion}
\label{sec:fw-expt}

In this section, 
we perform experiments on nonconvex low-rank matrix completion (Section~\ref{sec:noncvxFW}),
with square loss in \eqref{eq:lowrank2}.
The LSP regularizer is used,
with $\theta = \sqrt{\mu}$
as in \citep{qyao2015icdm}.
We use 
the \textit{MovieLens} data sets\footnote{\url{http://grouplens.org/data sets/movielens/}}
(Table~\ref{tab:recsys:dataset}),
which have been
commonly used for evaluating matrix completion 
\citep{hsieh2014nuclear,qyao2015icdm}.
They contain ratings $\{1,2,\dots,5\}$ assigned by various users on
movies.

\begin{table}[ht]
	\centering
	\caption{\textit{MovieLens} data sets used in the experiment.}
	\begin{tabular}{ c | c | c | c}
		\hline
		              & \#users & \#items & \#ratings  \\ \hline
		\textit{100K} & 943     & 1,682   & 100,000    \\ \hline
		\textit{1M}   & 6,040   & 3,449   & 999,714    \\ \hline
		\textit{10M}  & 69,878  & 10,677  & 10,000,054 \\ \hline
	\end{tabular}
	\label{tab:recsys:dataset}
\end{table}

The proposed Frank-Wolfe procedure 
(Algorithm~\ref{alg:novelfw}),
denoted N2C-FW, 
is compared with  the following algorithms:
\begin{enumerate}
\item FaNCL
\citep{qyao2015icdm}:
This is a recent nonconvex matrix regularization algorithm.
It is based on the proximal algorithm 
using efficient approximate SVD  and automatic thresholding of singular values.
	\item LMaFit \citep{wen2012solving}: 
	It factorizes $X$ as a product of low-rank matrices
	$U \in \R^{m \times k}$ and $V \in \R^{n \times k}$.
	The nonconvex objective $\frac{1}{2}\NM{P_{\Omega}(U V^{\top} - O)}{F}^2$
	is then 
	minimized by 
	alternating  minimization
	on $U$ and $V$
	using gradient descent.
	
	\item Active subspace selection (denoted ``active'') \citep{hsieh2014nuclear}: This
	solves the (convex) nuclear norm regularized problem (with $\kappa$ being the identity function in
	\eqref{eq:lowrank}) by using the active row/column subspaces
	to reduce the 
	optimization problem
	size.
\end{enumerate}
We do not compare with IRNN \citep{lu2014generalized} and GPG \citep{lu2015generalized}, which 
have been shown to be much slower than FaNCL \citep{qyao2015icdm}.

Following \citep{qyao2015icdm},
we use $50\%$ of the ratings for training, $25\%$ for validation and the rest for testing.
For performance evaluation, we use
(i) the testing RMSE; and
(ii) the recovered rank.
To reduce statistical variability, the experimental results are averaged
over 5 repetitions.

Results are shown in Table~\ref{tab:matcomp}.
As can be seen, the
nonconvex models (N2C-FW, FaNCL and LMaFit) achieve lower RMSEs than the convex model
(active), with N2C-FW having the smallest RMSE.
Moreover, the convex model needs a much higher rank than the nonconvex models,
	which agrees with the previous observations in \citep{mazumder2010spectral,qyao2015icdm}.
Thus, its running time is also much longer than the others.
Figure~\ref{fig:matcomp}
shows the convergence of the objective with CPU time.
As the recovered matrixs rank for the nonconvex models are very low (2 to 9 in 
Table~\ref{tab:matcomp}),
N2C-FW 
is much faster than the others as it
starts from a rank-one matrix and only increases its rank by one in each iteration.
Though 
FaNCL
uses singular value thresholding to truncate the SVD,
it does not control the rank as directly as N2C-FW and so is still slower.

\begin{table}[ht]
\centering
\caption{Results on the \textit{MovieLens} data sets.
	The best RMSE's (according to the pairwise t-test with 95\% confidence)
	are highlighted.}
\begin{tabular}{c|c|c|c|c}
	\hline
&        &      RMSE       & rank &     CPU time(sec)      \\ \hline
	\textit{100K} & N2C-FW & \textbf{0.855$\pm$0.004} &       2       &          \textbf{0.2$\pm$0.1}          \\ \cline{2-5}
   & FaNCL  &     0.857$\pm$0.003      &       2       &          0.4$\pm$0.1          \\ \cline{2-5}
   & LMaFit &     0.867$\pm$0.004      &       2       &          0.3$\pm$0.1          \\ \cline{2-5}
   & (convex) active &     0.875$\pm$0.002      &      52       &      1.8$\pm$0.1       \\ \hline
 \textit{1M}  & N2C-FW & \textbf{0.785$\pm$0.001} &       5       &  \textbf{9.3$\pm$0.1}  \\ \cline{2-5}
   & FaNCL  &     0.786$\pm$0.001      &       5       &      16.6$\pm$0.6      \\ \cline{2-5}
   & LMaFit &     0.812$\pm$0.002      &       5       &      14.7$\pm$0.7      \\ \cline{2-5}
   & (convex) active &     0.811$\pm$0.001      &      106      &      46.3$\pm$1.1      \\ \hline
	\textit{10M}  & N2C-FW & \textbf{0.778$\pm$0.001} &       9       & \textbf{313.0$\pm$6.6} \\ \cline{2-5}
   & FaNCL  &     0.779$\pm$0.001      &       9       &     615.7$\pm$13.2     \\ \cline{2-5}
   & LMaFit &     0.797$\pm$0.001      &       9       &     491.9$\pm$36.3     \\ \cline{2-5}
   & (convex) active &     0.808$\pm$0.001      &      137      &    1049.8$\pm$43.2     \\ \hline
\end{tabular}
\label{tab:matcomp}
\end{table}

\begin{figure}[ht]
	\centering
	\subfigure[\textit{MovieLens-100K}.]
	{\includegraphics[width = 0.45\textwidth]{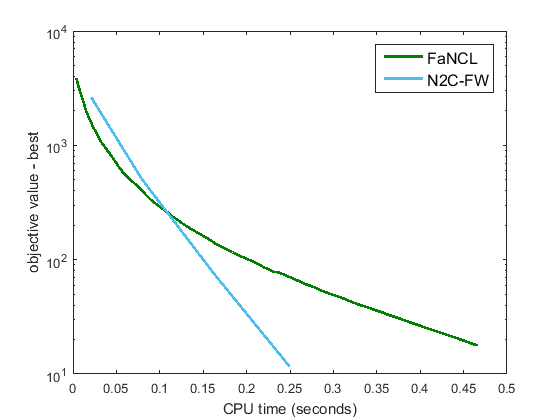}}
	\subfigure[\textit{MovieLens-1M}.]
	{\includegraphics[width = 0.45\textwidth]{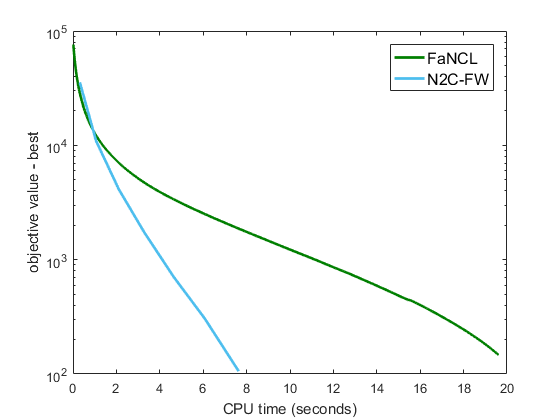}}
	\subfigure[\textit{MovieLens-10M}.]
	{\includegraphics[width = 0.45\textwidth]{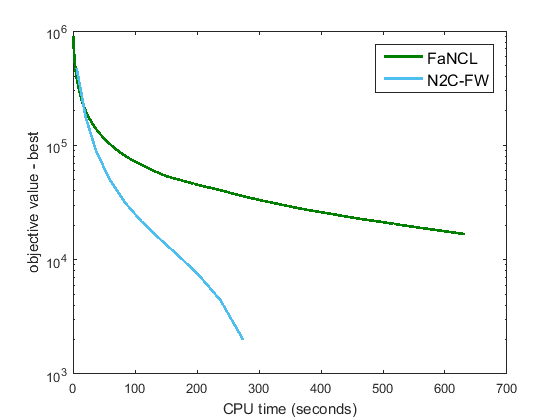}}
	\caption{Convergence of the objective vs CPU time on nonconvex low-rank matrix completion.
		The objectives 
		of LMaFit and active subspace selection  are different  from N2C-FW,
		and thus are 
		not shown.}
	\label{fig:matcomp}
\end{figure}


\subsection{Comparison with HONOR}
\label{sec:honor-expt}

In this section, we 
experimentally
compare the proposed method with 
HONOR (Section~\ref{sec:quasiNT})
on the  model in (\ref{eq:lasso}),
using the logistic loss and LSP regularizer.
Following \citep{gong2015honor}, 
we fix $\mu = 1$ in 
(\ref{eq:lasso}),
and $\theta$ in
the LSP regularizer to
$0.01 \mu$.
Experiments are performed on three large data sets,
\textit{kdd2010a}, \textit{kdd2010b} and \textit{url}
\footnote{\url{https://www.csie.ntu.edu.tw/~cjlin/libsvmtools/datasets/binary.html}}
(Table~\ref{tab:dataoglasso}).
Both \textit{kdd2010a} and \textit{kdd2010b} are educational data sets, and the task is to predict students' successful attempts to
answer concepts related to algebra. The
\textit{url} data set contains a collection of websites, and the task is to predict whether a
particular website is malicious.
We compare  
\begin{enumerate}
\item running HONOR  \citep{gong2015honor} directly on (\ref{eq:lasso}).
The threshold of the hybrid step  
in HONOR  
is set
to $10^{-10}$, which yields the best empirical performance in 
\citep{gong2015honor};
\item running
mOWL-QN \citep{gong2015modified}) on the transformed problem (\ref{eq:lasso:equiv}).
\end{enumerate}
To reduce statistical variability, the experimental results are averaged over 5 repetitions.

\begin{table}[ht]
	\centering
	\caption{Data sets used in the comparison with HONOR.}
	\begin{tabular}{c | c | c | c}
		\hline
		                   & \textit{kdd2010a} & \textit{kdd2010b} & \textit{url} \\ \hline
		number of samples  & 510,302           & 748,401           & 2,396,130    \\ \hline
		number of features & 20,216,830        & 29,890,095        & 3,231,961    \\ \hline
	\end{tabular}
\label{tab:dataoglasso}
\end{table}

As 
(\ref{eq:lasso}) and
(\ref{eq:lasso:equiv})
have the same optimization objective,
Figure~\ref{fig:honor} shows the
convergence of the objective
with CPU time.
As can be seen, mOWL-QN converges faster than HONOR.
This validates our claim that the curvature information of the nonconvex regularizer helps.

\begin{figure}[ht]
	\centering
	\subfigure[\textit{kdd2010a}.]
	{\includegraphics[width = 0.45\textwidth]{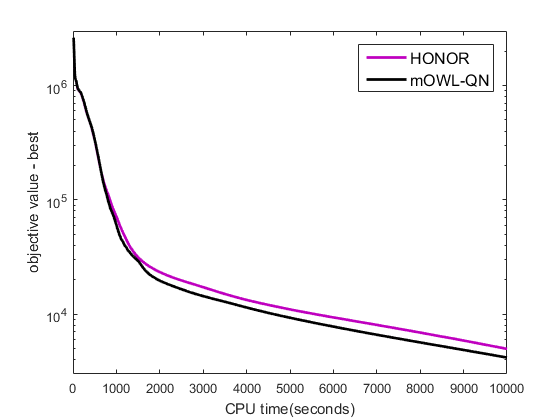}}
	\subfigure[\textit{kdd2010b}.]
	{\includegraphics[width = 0.45\textwidth]{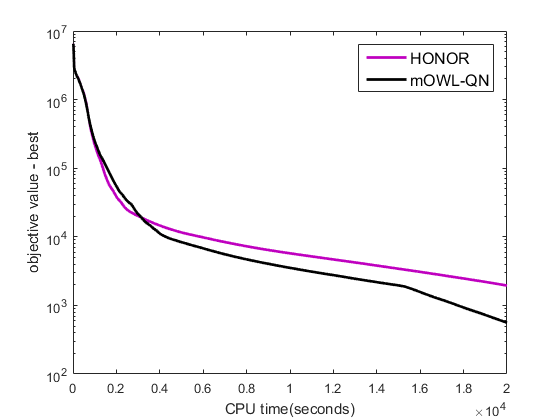}}
	\subfigure[\textit{url}.]
	{\includegraphics[width = 0.45\textwidth]{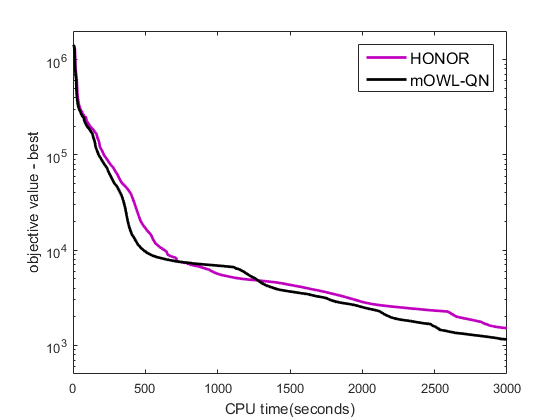}}
	\caption{Convergence of the objective vs CPU time 
		for HONOR 
		and mOWL-QN.}
	\label{fig:honor}
\end{figure}


\subsection{Image Denoising}
\label{sec:imgden}

In this section, we perform experiments on total variation image denoising with nonconvex loss and nonconvex
regularizer
(as introduced in Section~\ref{sec:tvdenoimg}).
The LSP function 
(with $\theta = 1$)
is used as $\kappa$ 
in \eqref{eq:ncvximgTV}
on both the loss and regularizer.
Eight popular images\footnote{\url{http://www.cs.tut.fi/~foi/GCF-BM3D/}}
from \citep{dabov2007image} 
are used
(Figure~\ref{fig:stdimage}).
They are then corrupted by
pepper-and-salt noise,
with $10\%$ of the pixels 
randomly 
set to $0$ or $255$ with equal probabilities.

For performance evaluation,
we use the 
$\text{RMSE} = \sqrt{\frac{1}{mn}\sum_{i = 1}^m \sum_{j = 1}^n (X_{ij} - \bar{X}_{ij})^2}$,
where $\bar{X} \in \R^{m \times n}$ is the clean image, and $X \in \R^{m \times n}$ is the recovered image.
To tune $\mu$, we pick the value that leads to the smallest RMSE on 
the first four  images
(\textit{boat}, \textit{couple}, \textit{fprint}, \textit{hill}).
Denoising performance is then reported 
on the
remaining images 
(\textit{house}, \textit{lena}, \textit{man}, \textit{peppers}).

\begin{figure}[ht]
\centering
\includegraphics[width = 0.115\textwidth]{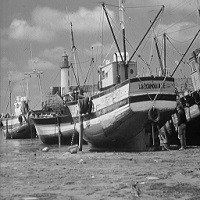}
\includegraphics[width = 0.115\textwidth]{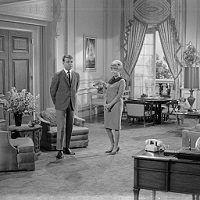}
\includegraphics[width = 0.115\textwidth]{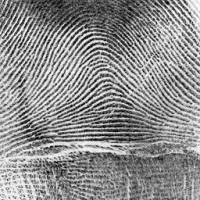}
\includegraphics[width = 0.115\textwidth]{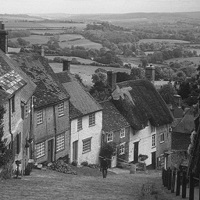}
\includegraphics[width = 0.115\textwidth]{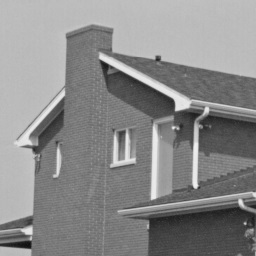}
\includegraphics[width = 0.115\textwidth]{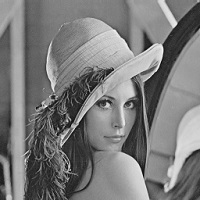}
\includegraphics[width = 0.115\textwidth]{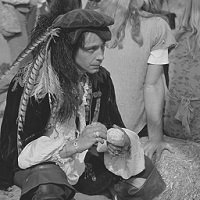}
\includegraphics[width = 0.115\textwidth]{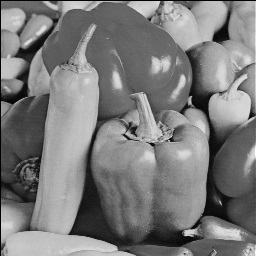}

\subfigure[\textit{boat}.]{\includegraphics[width = 0.115\textwidth]{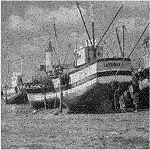}}
\subfigure[\textit{couple}.]{\includegraphics[width = 0.115\textwidth]{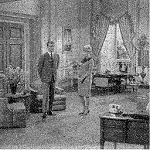}}
\subfigure[\textit{fprint}.]{\includegraphics[width = 0.115\textwidth]{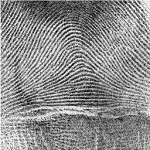}}
\subfigure[\textit{hill}.]{\includegraphics[width = 0.115\textwidth]{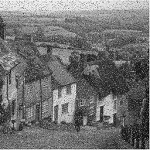}}
\subfigure[\textit{house}.]{\includegraphics[width = 0.115\textwidth]{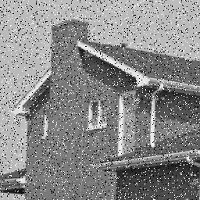}}
\subfigure[\textit{lena}.]{\includegraphics[width = 0.115\textwidth]{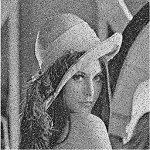}}
\subfigure[\textit{man}.]{\includegraphics[width = 0.115\textwidth]{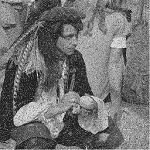}}
\subfigure[\textit{peppers}.]{\includegraphics[width = 0.115\textwidth]{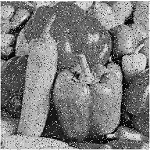}}

\caption{Samples images used in the denoising experiment.  Top: Clean images; Bottom: Noisy images.}
\label{fig:stdimage}
\end{figure}

The following algorithms will be compared:
\begin{enumerate}[itemsep = 0cm, topsep=0.125cm]
\item CCCP \citep{yuille2002concave}: 
	Proposition~\ref{pr:anotherDC} is used to construct DC decomposition for $\kappa$
	(Details are at Appendix~\ref{sec:app:CCCP});

\item Smoothing \citep{chen2012smoothing}:
	The nonsmooth $\kappa$ is smoothed,
	and then gradient descent is used
	(Details are at Appendix~\ref{sec:app:Smoothing});

\item nmAPG \citep{li2015accelerated}:
This optimizes	\eqref{eq:trantx1} with Algorithm~\ref{alg:nmAPG},
	and the exact proximal step is solved numerically using CCCP;
	
\item inexact-nmAPG: This optimizes \eqref{eq:imgTVtrans} with Algorithm~\ref{alg:inexactAPG}
	(with $\epsilon_t = 0.95^t$),
	and the inexact proximal step is solved numerically using L-BFGS.
	
	\item As a baseline, we also compare with ADMM \citep{boyd2011distributed} with the convex formulation.
\end{enumerate}

To reduce statistical variability, the experimental results are averaged over 5 repetitions.
The RMSE results are shown in Table~\ref{tab:imageRMSE}.
As can be seen, the (convex) ADMM formulation leads to the highest RMSE,
while CCCP, smoothing, nmAPG and inexact-nmAPG have the same RMSE which is lower than that of ADMM.
This agrees with previous observations that nonconvex formulations can yield better
performance than the convex ones.
Timing results are shown in 
Table~\ref{tab:imagetime} and Figure~\ref{fig:imagetime}.
As can be seen, smoothing has low iteration complexity but suffers from slow convergence.
CCCP and nmAPG both need to exactly solve a subproblem, and thus are also slow.
The inexact-nmAPG algorithm does not guarantee the objective value to be
monotonically decreasing as iteration proceeds. As the inexactness is initially large,
there is an initial spike in the objective. However, 
inexact-nmAPG then quickly converges, and is much faster than all the baselines.

\begin{table}[ht]
\centering
\caption{RMSE for image denoising.
The best RMSE's (according to the pairwise t-test with 95\% confidence) are highlighted.}
\label{tab:imageRMSE}
\begin{tabular}{c| c | c | c | c }
	\hline
	              & \textit{house}             & \textit{lena}              & \textit{man}               & \textit{peppers}           \\ \hline
	    CCCP      & \textbf{0.0205$\pm$0.0010} & \textbf{0.0174$\pm$0.0005} & \textbf{0.0223$\pm$0.0002} & \textbf{0.0207$\pm$0.0009} \\ \hline
	  smoothing   & \textbf{0.0205$\pm$0.0011} & \textbf{0.0174$\pm$0.0005} & \textbf{0.0223$\pm$0.0002} & \textbf{0.0207$\pm$0.0009} \\ \hline
	    nmAPG     & \textbf{0.0205$\pm$0.0010} & \textbf{0.0174$\pm$0.0005} & \textbf{0.0223$\pm$0.0002} & \textbf{0.0207$\pm$0.0009} \\ \hline
	inexact-nmAPG & \textbf{0.0205$\pm$0.0010} & \textbf{0.0174$\pm$0.0005} & \textbf{0.0223$\pm$0.0002} & \textbf{0.0207$\pm$0.0009} \\ \hline\hline
	(convex) ADMM & 0.0223$\pm$0.0011          & 0.0193$\pm$0.0005          & 0.0242$\pm$0.0002          & 0.0229$\pm$0.0008          \\ \hline
\end{tabular}
\end{table}

\begin{table}[ht]
\centering
\caption{CPU time (seconds) for image denoising.
	The shortest CPU time (according to the pairwise t-test with 95\% confidence) are highlighted.}
\label{tab:imagetime}
\begin{tabular}{c| c | c | c | c }
	\hline
	              & \textit{house}        & \textit{lena}         & \textit{man}          & \textit{peppers}     \\ \hline
	    CCCP      & 21.0$\pm$2.3          & 270.0$\pm$13.0        & 325.3$\pm$17.4        & 14.5$\pm$1.2         \\ \hline
	  smoothing   & 75.5$\pm$2.0          & 433.1$\pm$4.8         & 437.7$\pm$6.8         & 61.9$\pm$1.7         \\ \hline
	    nmAPG     & 19.4$\pm$2.3          & 91.4$\pm$7.3          & 104.4$\pm$2.7         & 16.1$\pm$1.8         \\ \hline
	inexact-nmAPG & \textbf{10.3$\pm$1.1} & \textbf{37.9$\pm$5.0} & \textbf{43.0$\pm$7.6} & \textbf{8.1$\pm$0.2} \\ \hline\hline
	(convex) ADMM & 3.0$\pm$0.1           & 42.8$\pm$1.1          & 46.9$\pm$1.0          & 2.2$\pm$0.1          \\ \hline
\end{tabular}
\end{table}

\begin{figure}[ht]
	\centering
	\subfigure[\textit{house}.]
	{\includegraphics[width = 0.45\textwidth]{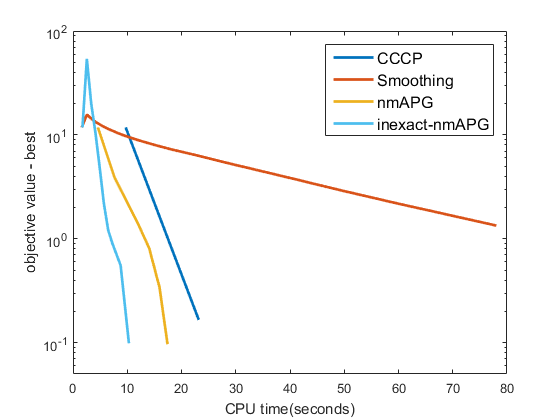}}
	\subfigure[\textit{lena}.]
	{\includegraphics[width = 0.45\textwidth]{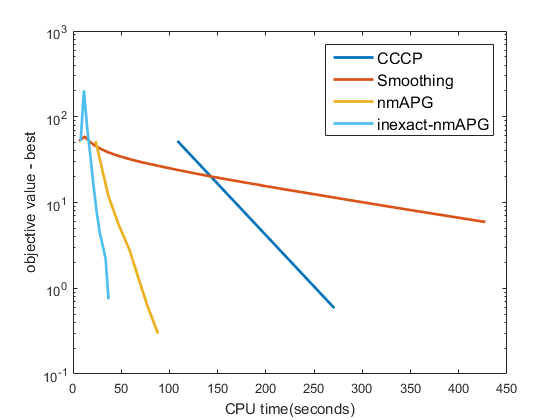}}
	\subfigure[\textit{man}.]
	{\includegraphics[width = 0.45\textwidth]{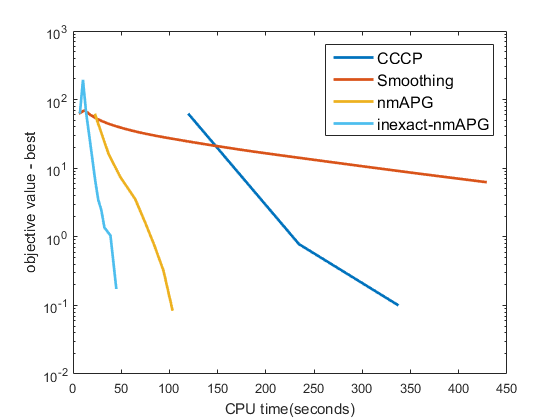}}
	\subfigure[\textit{peppers}.]
	{\includegraphics[width = 0.45\textwidth]{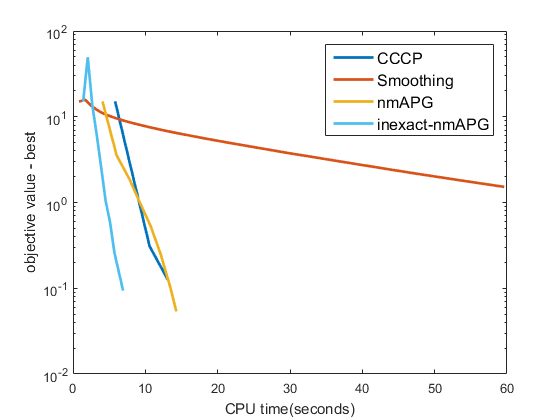}}
	\caption{CPU time (seconds) vs objective value on different images.}
	\label{fig:imagetime}
\end{figure}


\section{Conclusion}
\label{sec:conc}

In this paper, 
we proposed  a novel approach to learning with nonconvex regularizers.
By
moving the nonconvexity associated with the 
nonconvex regularizer
to the loss,  the 
nonconvex regularizer is convexified to become a familiar convex regularizer while the augmented
loss is still Lipschitz smooth. 
This allows one to reuse efficient algorithms originally designed for convex regularizers 
on the transformed problem.
To illustrate usages with the proposed transformation, 
we plug it into many popular optimization algorithms.
First,
we consider the proximal algorithm,
and showed that 
while 
the proximal step 
is expensive on the original problem,
it becomes much easier on the transformed problem.
We further propose an inexact proximal algorithm,
which allows inexact update of proximal step when it does not have a closed-form solution.
Second,
we combine the proposed convexification scheme with the Frank-Wolfe algorithm on learning
low-rank matrices, and showed that
its crucial linear programming step becomes cheaper and more easily solvable.
As no convergence results exist on this nonconvex problem, 
we designed a novel Frank-Wolfe algorithm based on the proposed transformation and with convergence guarantee.
Third,
when using with ADMM and SVRG, 
we showed that the existing convergence results can be applied on the transformed problem but
not on the original one.
We further extend the proposed transformation to handle nonconvex and nonsmooth loss functions,
and illustrate its benefits on the total variation model and robust sparse coding.
Finally,
we demonstrate  the empirical advantages of working with the transformed problems on various
tasks with both synthetic and real-world data sets. Experimental results show that better performance 
can be obtained with nonconvex regularizers,
and algorithms on the transformed problems run much faster than the state-of-the-art on the original problems.

\appendix
\section{Proofs}
\label{sec:proof}


\subsection{Proposition~\ref{prop:smooth}}
\label{app:keyobv}



\begin{proof}
First, we introduce a few Lemmas.

\begin{lemma}\citep{golub2012matrix}
For $x \neq 0$, the gradient of the $\ell_2$-norm is $\nabla_{x_i} \NM{x}{2} = x_i/\NM{x}{2}$.
	\label{lem:gradnorm}
\end{lemma}

Let $h(z) = \kappa(\NM{z}{2}) - \kappa_0 \NM{z}{2}$. 
\begin{lemma}\label{lem:lem1}
\begin{equation} \label{eq:app1}
	\nabla_{z_i} h(z) = \left\{ \begin{array}{ll}
	\frac{\kappa'(\NM{z}{2}) - \kappa_0}{\NM{z}{2}} z_i & \text{if}\; z \neq 0\\
	0 & \text{otherwise}
	\end{array} \right..
	\end{equation}
\end{lemma}

\begin{proof}
For $z \neq 0$,
$\NM{z}{2}$ is differentiable (Lemma~\ref{lem:gradnorm}), 
and we obtain the first part of (\ref{eq:app1}).
For $z =0$,
let
$\bar{h}_i(z) = \frac{\kappa'(\NM{z}{2}) - \kappa_0}{\NM{z}{2}} z_i$.
Consider any $\Delta$ with $\NM{\Delta}{2} = 1$.
\begin{align*}
\lim_{\alpha \rightarrow 0^+}{\bar{h}_i(0 + \alpha \Delta)}
& = \lim_{\alpha \rightarrow 0^+} \frac{\kappa'(\NM{\alpha \Delta}{2}) - \kappa_0}
{\NM{\alpha \Delta}{2}} \alpha \Delta_i,
\\
& = \lim_{\alpha \rightarrow 0^+} (\kappa'(\alpha) - \kappa_0)\Delta_i = 0,
\end{align*}
as $\lim_{\alpha \rightarrow 0^+}\kappa'(\alpha) - \kappa_0 = 0$.
Thus, $h(z)$ is smooth at $z = 0$, and we obtain the second part of \eqref{eq:app1}.
\end{proof}





\begin{lemma} \citep{eriksson2013applied} \label{pr:Lipschitz2}
Let $f : \R \rightarrow \R$ be a differentiable function.
(i)
If its derivative $f'$ is bounded,
then $f$ is Lipschitz-continuous with constant equal to the maximum value of $|f'|$.
\end{lemma}

\begin{lemma} \citep{eriksson2013applied} \label{pr:Lipschitz3}
If a continuous function $f:\R \rightarrow \R$ is $L_1$-Lipschitz continuous in $[a, b]$ and $L_2$-Lipschitz continuous in $[b, c]$ 
(where $- \infty \le a < b < c \le \infty$),
then it is $\max(L_1, L_2)$-Lipschitz continuous in $[a, c]$.
\end{lemma}

\begin{lemma} \label{lem:bar-h}
Let $z$ be an arbitrary vector, and $e_i$ be the unit vector with only its $i$th dimension equal to $1$.
Define $\hat{h}_i(\gamma) = \frac{\kappa'(\NM{z + e_i \gamma}{2}) - \kappa_0}{\NM{z + e_i \gamma}{2}} (z_i + \gamma)$.
Then, $\hat{h}$ is $2\rho$-Lipschitz continuous.
\end{lemma}

\begin{proof}
Since $\kappa'$ is non-differentiable only at finite points,
and let them be $\{ \hat{\alpha}_1, \dots, \hat{\alpha}_k \}$ where $\hat{\alpha}_1 < \dots < \hat{\alpha}_k$.
We partition $(-\infty, \infty)$ into intervals $(-\infty, \hat{\alpha}_1] \cup [\hat{\alpha}_1, \hat{\alpha}_2] \cup \cdots \cup [\hat{\alpha}_k, \infty)$,
such that 
$\kappa''$ exists
in each interval.
Let $w = z + e_i \gamma$.
For any interval,
\begin{align}
\hat{h}_i'(\gamma)
= 
\frac{\kappa''(\NM{w}{2})}{\NM{w}{2}} (z_i + \gamma)^2
+ \left(1 - \frac{(z_i + \gamma)^2}{\NM{w}{2}^2} \right)   
\frac{\kappa'(\NM{w}{2}) - \kappa_0}{\NM{w}{2}}.
\label{eq:temp9}
\end{align}

Let $\phi(\alpha) = \kappa'(\alpha) - \kappa_0$, where $\alpha \ge 0$.
Note that 
$\phi(0) = 0$. Moreover,
$\phi(\alpha)$
is $\rho$-Lipschitz continuous as $\kappa$ is $\rho$-Lipschitz smooth. Thus,
\begin{align*}
|\phi(\alpha) - \phi(0)|
= |\kappa'(\alpha) - \kappa_0| \le \rho \alpha,
\end{align*}
and so
\begin{align}
\left| \kappa'(\NM{w}{2}) - \kappa_0 \right|  \le \rho \NM{w}{2}.
\label{eq:temp18}
\end{align}
Note that $(z_i + \gamma)^2 \le \NM{w}{2}^2$, \eqref{eq:temp9} can be written as
\begin{align*}
\left| \hat{h}_i'(\gamma) \right| 
& \le 
\left| \frac{\kappa''(\NM{w}{2})}{\NM{w}{2}} (z_i + \gamma)^2 \right| 
+ \left| \left(1 - \frac{(z_i + \gamma)^2}{\NM{w}{2}^2} \right)   
\frac{\kappa'(\NM{w}{2}) - \kappa_0}{\NM{w}{2}} \right| 
\\
& \le 
\left|\kappa''(\NM{w}{2})\right| 
+ \left| \frac{\kappa'(\NM{w}{2}) - \kappa_0}{\NM{w}{2}} \right| 
\le 2 \rho,
\end{align*}
where the last inequality is due to that $\kappa$ is $\rho$-Lipschitz smooth and \eqref{eq:temp18}.
Thus, $|\hat{h}_i'(\gamma)| \le 2 \rho$,
and by Lemma~\ref{pr:Lipschitz2}, we have $\hat{h}_i(\gamma)$ is $2\rho$-Lipschitz continuous on any
interval.
Obviously $\hat{h}_i$ is continuous,
and we conclude that $\hat{h}_i$ is also $2\rho$-Lipschitz continuous by Lemma~\ref{pr:Lipschitz3}. 
\end{proof}

From Lemma~\ref{lem:bar-h},
$\hat{h}_i$ is 
$2\rho$-Lipschitz continuous.
Thus, $\nabla h$ is 
$2\rho$-Lipschitz continuous 
in each of its dimensions. For any 
$x, y\in \R^d$,
\begin{eqnarray*}
	\NM{\nabla h(x) - \nabla h(y)}{2}^2
	& = & \sum_{i = 1}^d \left[  \nabla_{x_i} h(x) - \nabla_{y_i} h(y) \right]^2 \\
	& \le & 4\rho^2 \sum_{i = 1}^d (x_{i} - y_{i})^2
	= 4\rho^2 \NM{x- y}{2}^2,
\end{eqnarray*}
and hence
$h$ is $2\rho$-Lipschitz smooth.

Finally, we will show that $h(z)$ is also concave.
\begin{lemma}\citep{boyd2004convex}
	\label{lem:concave}
	$\phi(x) = \pi(q(x))$ is concave
	if $\pi$ is concave, non-increasing and $q$ is convex.
\end{lemma}

Let $\pi(\alpha) = \kappa(\alpha) - \kappa_0 \alpha$, where $\alpha \ge 0$.
Note that $\pi$ is concave.
Moreover, $\pi(0) = 0$ and $\pi'(\alpha) \le 0$.
Thus, $\pi(\alpha)$ is non-increasing on $\alpha \ge 0$.
Next, let $q(z) = \NM{z}{2}$.
Then,
$h(z) \equiv \kappa(\NM{z}{2}) - \kappa_0 \NM{z}{2} = \pi(q(z))$.
As $q$ is convex,
$h(z)$ is concave
from Lemma~\ref{lem:concave}.
\end{proof}


\subsection{Corollary~\ref{cor:smooth}}

\begin{proof}
From Proposition~\ref{prop:smooth}
and definition of $\bar{g}_i$,
we can see it is concave.
Then,
for any $x, y$, 
\begin{align*}
\NM{\nabla h(A_i x) - \nabla h(A_i y)}{2}^2
\le 4 \rho^2 \NM{A_i x - A_i y}{2}^2
\le 4 \rho^2 \NM{A_i}{F}^2 \NM{x - y}{2}^2.
\end{align*}
Thus, $\bar{g}_i$ is $2\rho\NM{A_i}{F}$-Lipschitz smooth.
\end{proof}


\subsection{Corollary~\ref{cor:c1}}

\begin{proof}
It is easy to see that $\breve{g}(x) = \kappa_0 \sum_{i = 1}^K \mu_i \NM{A_i x}{2}$ is convex but not smooth.
Using Corollary~\ref{cor:smooth},
as each $\bar{g}_i$ is concave and Lipschitz-smooth,
$\bar{g}$ is also concave and Lipschitz-smooth.
\end{proof}


\subsection{Proposition~\ref{prop:c3}}

\begin{proof}
First,
we introduce a few lemmas.

\begin{definition} \citep{bertsekas1999nonlinear}
A function $f:\R^m \rightarrow \R$ 
is {\em absolute symmetric}
if $f\left( \left[ x_1; \dots; x_m \right]\right) = f\left(\left[ |x_{\pi(1)}|; \dots; |x_{\pi(m)}| \right]\right)$
for any permutation $\pi$.
\end{definition}

\begin{lemma}\citep{lewis2005nonsmooth} \label{lem:svscvx}
Let $\sigma(X) = [\sigma_1(X); \dots; \sigma_m(X)]$ be the vector containing singular values of $X$.  For an absolute symmetric function $f: \R^m \rightarrow \R$,
	$\phi(X) \equiv f(\sigma(X))$ is concave on $X$
	if and only if $f$ is concave.
\end{lemma}

From the definition of
$\bar{g}$ in (\ref{eq:barg}), 
\begin{align*}
\bar{g}(X)
= \bar{\mu}  \sum_{i = 1}^m
( \kappa(\sigma_i(X)) - \kappa_0 \NM{X}{*} )
= \bar{\mu} \sum_{i = 1}^m
\left( \kappa(\sigma_i(X)) - \kappa_0 \sigma_i(X) \right).
\end{align*}
Let 
\begin{equation} \label{eq:h}
h(x) = \bar{\mu} \sum_{i = 1}^m (\kappa(|x_i|) - \kappa_0 |x_i|).
\end{equation}
Obviously, 
$h$ is
absolute symmetric.
From Remark~\ref{remark:c2},
$h$ is concave. Thus, $\bar{g}$ is also concave by Lemma~\ref{lem:svscvx}.


\begin{lemma}\citep{lewis2005nonsmooth} \label{lem:gradspec}
Let the SVD of $X$ 
be $U \Diag{\sigma(X)} V^{\top}$,
where $\sigma(X) = \left[ \sigma_1(X); \dots; \sigma_m(X) \right]$,
$f:\R^m \rightarrow \R$ be smooth and absolute symmetric, and $\phi(X) \equiv f(\sigma(X))$. We have 
\begin{enumerate}
\item $\nabla \phi(X) = U \Diag{ \nabla f(\sigma(X)) } V^{\top}$; and
\item If $f$ is $L$-Lipschitz smooth, then $\phi$ is also $L$-Lipschitz smooth.
\end{enumerate}
\end{lemma}
 
From Remark~\ref{remark:c2},
$h$ in (\ref{eq:h}) is $2\rho$-Lipschitz smooth. Hence,
from Lemma~\ref{lem:gradspec},
$\bar{g}(X)$ is also  $2\rho$-Lipschitz smooth
and $\nabla \bar{g} (X) = U \Diag{\nabla h(\sigma(X))} V^{\top}$.
\end{proof}


\subsection{Proposition~\ref{pr:anotherDC}}

\begin{proof}
First, we introduce the following lemma.
	
\begin{lemma}\citep{boyd2004convex}
	\label{lem:convex}
	$\phi(x) = \pi(q(x))$ is convex
	if $\pi$ is convex, non-decreasing and $q$ is convex.
\end{lemma}

Let $q(x) = \NM{x}{2}$,
and $\pi(\alpha) = \kappa(\alpha) + \frac{\rho}{2} \alpha^2$ where $\alpha \ge 0$.
Thus, $\phi(x) = \pi(q(x)) = \kappa(\NM{x}{2}) + \frac{\rho}{2} \NM{x}{2}^2$.
Obviously, $q$ is convex.
For $\alpha \ge \beta \ge 0$, $0\le \kappa'(\alpha) \le \kappa'(\beta)$.
As $\kappa$ is $\rho$-Lipschitz smooth,
$\kappa'(\beta) - \kappa'(\alpha) \le \rho(\alpha - \beta)$.
Thus,
$\pi'(\alpha) - \pi'(\beta) 
= \kappa'(\alpha) + \rho \alpha - \kappa'(\beta) - \rho \beta \ge 0$,
i.e., $\pi$ is convex.
Besides, $\pi'(0) = \kappa'(0) \ge 0$.
Thus, $\pi'(\alpha) \ge 0$ and $\pi$ is also non-decreasing.
By Lemma~\ref{lem:convex}, 
$\phi$ is also convex.
\end{proof}


\subsection{Theorem~\ref{thm:convAPG}}

\begin{proof} 
First, we introduce a few lemmas.

\begin{lemma} \label{lem:funcvale}
Let $\tilde{X}$ be an inexact solution of the proximal step
$\min_Z h(Z)$, where 
$h(Z)=
\frac{1}{2} \NM{Z - (X - \frac{1}{\tau} \nabla \bar{f}(X))}{F}^2 + \frac{1}{\tau} \breve{g}(Z)$.
Let $\hat{X} = \arg\min_Z h(Z)$.
If $h(\tilde{X}) - h(\hat{X}) \le \epsilon$, 
then
\begin{align*}
F(\tilde{X}) 
\le F(X) - \frac{\tau - \bar{L}}{2} \NM{\tilde{X} - X}{F}^2 
+ \tau\epsilon.
\end{align*}
\end{lemma}

\begin{proof}
Let $\phi(Z) = \left\langle Z - X, \nabla f(X) \right\rangle 
+ \frac{\tau}{2} \NM{Z - X}{F}^2
+ \breve{g}(Z)$.
We have
\begin{eqnarray}
\hat{X} & = & \arg\min_Z h(Z) = \arg\min_Z \phi(Z), \label{eq:temp13} \\
\phi(Z) & = & \tau h(Z) - \frac{1}{\tau}\NM{\nabla \bar{f}(X)}{F}^2.
\label{eq:temp14}
\end{eqnarray}
From \eqref{eq:temp13}, we have
\begin{align}
\phi(\hat{X}) 
= \langle  \hat{X} - X, \nabla f(X) \rangle 
+ \frac{\tau}{2} \NM{\hat{X} - X}{F}^2
+ \breve{g}(\hat{X})
\le \breve{g}(X).
\label{eq:temp15}
\end{align}
As $h(\tilde{X}) - h(\hat{X}) \le \epsilon$, 
from \eqref{eq:temp14} (note that $\NM{\nabla \bar{f}(X)}{F}^2$ is a constant), we have
\begin{align*}
\phi(\tilde{X}) - \phi(\hat{X})
= \tau ( h(\tilde{X}) - h(\hat{X}) )
 \le \tau \epsilon
\end{align*}
Then with \eqref{eq:temp15},
we have
$\phi(\tilde{X}) \le \tau \epsilon + \phi(\hat{X}) \le \breve{g}(X) + \tau \epsilon$,
i.e.,
\begin{align}
\langle  \tilde{X} - X, \nabla f(X) \rangle 
+ \frac{\tau}{2} \NM{\tilde{X} - X}{F}^2
+ \breve{g}(\tilde{X})
\le \breve{g}(X) + \tau \epsilon.
\label{eq:temp8}
\end{align}
As $\bar{f}$ is $\bar{L}$-Lipschitz smooth,
\begin{align*}
\bar{f}(\tilde{X})
\le \bar{f}(X) + \langle  \tilde{X} - X, \nabla f(X) \rangle + \frac{\bar{L}}{2}\NM{\tilde{X} - X}{F}^2.
\end{align*}
Combining with \eqref{eq:temp8},
we obtain
\begin{align*}
\bar{f}(\tilde{X}) 
+ \frac{\tau}{2} \NM{\tilde{X} - X}{2}^2
+ \breve{g}(\tilde{X})
\le 
\bar{f}(X) + \frac{\bar{L}}{2}\NM{\tilde{X} - X}{F}^2
+ \breve{g}(X) + \tau \epsilon.
\end{align*}
Thus, 
$F(\tilde{X}) \le F(X) - \frac{\tau - \bar{L}}{2} \NM{\tilde{X} - X}{F}^2 + \tau \epsilon$.
\end{proof}

If step~6 in Algorithm~\ref{alg:inexactAPG} is satisfied, 
$X_{t + 1} = \tilde{Z}_{t + 1}$,
and 
\begin{align}
F(X_{t + 1}) \le F(X_t) - \frac{\delta}{2}\NM{X_{t + 1} - Y_t}{F}^2.
\label{eq:temp10}
\end{align}
Otherwise, step~9 is executed, and from Lemma~\ref{lem:funcvale},
we have
\begin{align}
F(X_{t + 1}) 
\le F(X_t) - \frac{\tau - \bar{L}}{2} \NM{X_{t + 1} - X_t}{F}^2 + \tau \epsilon_t.
\label{eq:temp11}
\end{align}

Partition $\Omega(T) = \{ 1, 2, \dots, T \}$ into $\Omega_1(T)$ and $\Omega_2(T)$, such that step~7 is performed if $t \in \Omega_1(T)$; 
and execute
step~9 
otherwise.
Combining \eqref{eq:temp10} and \eqref{eq:temp11},
we have
\begin{align}
F(X_1) & - F(X_{T + 1})
\notag \\
& \ge \frac{\delta}{2} \sum_{t \in \Omega_1(T)} \NM{X_{t + 1} - Y_{t}}{F}^2
+ \frac{\tau - \bar{L}}{2} \sum_{t \in \Omega_2(T)} 
\left( \NM{X_{t + 1} - X_{t}}{F}^2 - \tau \epsilon_{t} \right),
\notag \\
& \ge \frac{\delta}{2} \sum_{t \in \Omega_1(T)} \NM{X_{t + 1} - Y_{t}}{F}^2
+ \frac{\tau - \bar{L}}{2} \sum_{t \in \Omega_2(T)} 
\NM{X_{t + 1} - X_{t}}{F}^2
- \frac{(\tau - \bar{L})\tau}{2} \sum_{t \in \Omega_2(T)} \epsilon_t
\notag \\
& \ge \frac{\delta}{2} \sum_{t \in \Omega_1(T)} \NM{X_{t + 1} - Y_{t}}{F}^2
+ \frac{\tau - \bar{L}}{2} \sum_{t \in \Omega_2(T)} 
\NM{X_{t + 1} - X_{t}}{F}^2
- \frac{(\tau - \bar{L})\tau}{2} \sum_{t = 1}^{\infty} \epsilon_t
\notag \\
& \ge \frac{\delta}{2} \sum_{t \in \Omega_1(T)} \NM{X_{t + 1} - Y_{t}}{F}^2
- c_1 + \frac{\tau - \bar{L}}{2} \sum_{t \in \Omega_2(T)} \NM{X_{t + 1} - X_{t}}{F}^2,
\label{eq:temp12} 
\end{align}
where $c_1 = \frac{(\tau - \bar{L})\tau}{2} \sum_{t = 1}^{\infty} \epsilon_t < \infty$ and $c_1 \ge 0$.
From \eqref{eq:temp12}, we have
\begin{align}
F(X_1) - \inf_X F(X) 
& + c_1
\ge F(X_1) - \lim_{T \rightarrow \infty} F(X_{T + 1}) + c_1 \nonumber
\\
\ge & \lim_{T \rightarrow \infty}
\frac{\delta}{2} \!\!\! \sum_{t \in \Omega_1(T)} \NM{X_{t + 1} - Y_{t}}{F}^2
+ \frac{\tau - \bar{L}}{2} \!\!\! \sum_{t \in \Omega_2(T)} \NM{X_{t + 1} - X_{t}}{F}^2
\equiv c_2. \label{eq:temp16}
\end{align}
From Assumption A1,
$c_2 \le F(X_1) - \inf_X F(X) + c_1  < \infty$,
thus $c_2 \ge 0$ is a finite constant.
Let $\Omega^{\infty}_1 = \lim_{T \rightarrow \infty} \Omega_1(T)$,
and $\Omega^{\infty}_2 = \lim_{T \rightarrow \infty} \Omega_2(T)$.
Consider the three cases:
\begin{enumerate}
\item \underline{$|\Omega^{\infty}_1|$ is finite,
and $|\Omega^{\infty}_2|$ is infinite}.
As $|\Omega^{\infty}_2| = \infty$ and $\lim_{\NM{X}{F} \rightarrow \infty} F(X) = \infty$ from Assumption A1 and \eqref{eq:temp16}, we must have
\begin{align*}
\lim_{t \in \Omega^{\infty}_2, t \rightarrow \infty}
\NM{X_{t + 1} - X_t}{F}^2 = 0.
\end{align*}
Thus, 
there exists a limit point
such that $X_* = \lim_{t_j \in \Omega^{\infty}_2, t_j \rightarrow \infty} X_{t_j}$ for a subsequence $\{ X_{t_j} \}$ of $\{ X_t \}$.
Since $\lim_{t_j \rightarrow \infty} \epsilon_{t_j} = 0$, then
\begin{align*}
\lim_{t_j \in \Omega^{\infty}_2, t_j \rightarrow \infty} X_{t_j + 1} 
= \lim_{t_j \in \Omega^{\infty}_2, t_j \rightarrow \infty} \Px{\frac{1}{\tau} \breve{g}}{X_{t_j} - \frac{1}{\tau} \nabla \bar{f}(X_{t_j})}.
\end{align*}
As a result, 
\begin{align*}
0 \in \lim_{t_j \in \Omega^{\infty}_2, t_j \rightarrow \infty} 
\frac{1}{\tau}\nabla \bar{f}(X_{t_j}) + (X_{t_j + 1} - X_{t_j}) 
+ \frac{1}{\tau} \partial \breve{g}(X_{t_j + 1}).
\end{align*}
Since both $\lim_{t_j \in \Omega^{\infty}_2, t_j \rightarrow \infty}X_{t_j} = \lim_{t_j \in \Omega^{\infty}_2, t_j \rightarrow \infty} X_{t_j + 1} = X_*$, 
we then have
$\nabla \bar{f}(X_*) + \partial \breve{g}(X_*) \ni 0$,
and $X_*$ is a critical point of \eqref{eq:pro}.

\item \underline{$|\Omega^{\infty}_1|$ is infinite,
	and $|\Omega^{\infty}_2|$ is finite.}
As $\Omega^{\infty}_1$ is infinite and $\lim_{\NM{X}{F} \rightarrow \infty} F(X) = \infty$ from Assumption A1 and \eqref{eq:temp16}, we must have
\begin{align*}
\lim_{t_j \in \Omega^{\infty}_1, t_j \rightarrow \infty}
\NM{X_{t_j + 1} - Y_{t_j}}{F}^2 = 0.
\end{align*}
for a subsequence $\{ X_{t_j} \}$ of $\{ X_t \}$.
Thus, there must exist a limit point such that 
\begin{align}
X_* 
= \lim_{t_j \in \Omega^{\infty}_1, t_j \rightarrow \infty} X_{t_j + 1} 
=  \lim_{t_j \in \Omega^{\infty}_1, t_j \rightarrow \infty} Y_{t_j}.
\label{eq:temp19}
\end{align}
As $\lim_{t_j \rightarrow \infty} \epsilon_{t_j} = 0$,  we have
\begin{align*}
0 \in \lim_{t_j \in \Omega^{\infty}_1, t_j \rightarrow \infty} 
\frac{1}{\tau}\nabla \bar{f}(Y_{t_j}) + (X_{t_j + 1} - Y_{t_j}) 
+ \frac{1}{\tau} \partial \breve{g}(X_{t_j + 1}).
\end{align*}
From \eqref{eq:temp19}, thus we have
$\nabla \bar{f}(X_*) + \partial \breve{g}(X_*) \ni 0$
and $X_*$ is a critical point of \eqref{eq:pro}.

\item \underline{Both $\Omega^{\infty}_1$ and $\Omega^{\infty}_2$ are infinite.}
From above two cases, 
we can see $\{X_t\}$ is bounded and, the limit points of $\{X_t\}$ are also critical points either $|\Omega^{\infty}_1|$ or $|\Omega^{\infty}_2|$ is infinite.
In the third case, both of them are infinite,
thus any limit points of $\{X_t\}$ are also critical points of \eqref{eq:pro}.
\end{enumerate} 

As a result,
$\{ X_t \}$ are bounded and its limits points are all critical points of \eqref{eq:pro}.
\end{proof}


\subsection{Proposition~\ref{pr:inexactrate}}

\begin{proof}
From \eqref{eq:temp16},	
we have
\begin{align}
\frac{\delta}{2} \sum_{t_1 \in \Omega_1(T)} \NM{X_{t_1 + 1} - Y_{t_1}}{F}^2
+ \frac{\tau - \bar{L}}{2} \sum_{t_2 \in \Omega_2(T)} \NM{X_{t_2 + 1} - X_{t_2}}{F}^2 < c_2,
\label{eq:temp17}
\end{align}
where $c_2 \in (0, \infty)$ is a positive constant.
Let $c_3 = \min(\frac{\delta}{2}, \frac{\tau - \bar{L}}{2})$ and
using the definition of $V_t$,
\eqref{eq:temp17} can be written as  
\begin{align*}
c_3 \sum_{t = 1}^T \NM{X_{t + 1} - V_t}{F}^2 
\le
\frac{\delta}{2} \sum_{t_1 \in \Omega_1(T)} \NM{X_{t_1 + 1} - Y_{t_1}}{F}^2
+ \frac{\tau - \bar{L}}{2} \sum_{t_2 \in \Omega_2(T)} \NM{X_{t_2 + 1} - X_{t_2}}{F}^2
\le c_2.
\end{align*}
Since $c_2$ is finite, thus $\lim_{t \rightarrow \infty} d_t \equiv \NM{X_{t + 1} - V_t}{F}^2 = 0$.
Besides, we have
\begin{align*}
\min_{t = 1, \dots, T} \sum_{t = 1}^T \NM{X_{t + 1} - V_t}{F}^2
\le \frac{1}{T} \sum_{t = 1}^T \NM{X_{t + 1} - V_t}{F}^2
\le \frac{c_2}{c_3 T}.
\end{align*}
\end{proof}


\subsection{Proposition~\ref{pr:gradkappa}}

\begin{proof}
Note from (\ref{eq:lowf}) that $\nabla \bar{f}(S) = \nabla f(S) + \nabla \bar{g}(S)$.
Using the matrix chain rule,
since $S = \alpha X_t + \beta u_t v_t^{\top}$ and $\frac{\partial S}{\partial \alpha} = X_t$,
then
\begin{align*}
\frac{\partial \bar{f}(S)}{\partial \alpha} 
= \left\langle \nabla \bar{f}(S), \frac{\partial S}{\partial \alpha} \right\rangle
= \alpha \langle X_t, \nabla \bar{f}(S) \rangle.
\end{align*}
Similarly,
since $\frac{\partial S}{\partial \beta} = u_t v_t^{\top}$
\begin{align*}
\frac{\partial \bar{f}(S)}{\partial \beta} 
= \left\langle \nabla \bar{f}(S), \frac{\partial S}{\partial \beta} \right\rangle
= \beta \left\langle u_t v_t^{\top}, \nabla \bar{f}(S) \right\rangle = \beta \left( u_t^{\top} \nabla \bar{f}(S) v_t \right).
\end{align*}
As $\bar{g}(S)  = \mu \sum_{i = 1}^m \kappa(\sigma_i(S)) - \mu \kappa_0 \sigma_i(S)$, 
using Lemma~\ref{lem:gradspec}, 
$\nabla \bar{f}(X) 
= \nabla f(S) + \mu U_S \Diag{w} V_S^{\top}$
and $w_i = \kappa'(\sigma_i(S)) - \kappa_0$.
\end{proof}


\subsection{Corollary~\ref{pr:gradkappa2}}

\begin{proof}
Note that the SVD of $X$ is $(U U_B) \Diag{[\sigma_1(B), \dots, \sigma_k(B)]} (V V_B)^{\top}$.
Using Lemma~\ref{lem:gradspec}, 
\begin{align*}
\nabla \bar{f}(X)
= \nabla f(X) + \nabla \bar{g}(X)
= \nabla f(X) + \mu (U U_B) \Diag{w} (V V_B)^{\top}.
\end{align*}
where $w \in \R^k$ with $w_i = \kappa'(\sigma_i(B)) - \kappa_0$.
\end{proof}


\subsection{Proposition~\ref{pr:svreduce}}

\begin{proof}
As $\bar{g}(X)$ is defined on singular values of the input matrix $X$,
we only need to show $U B V^{\top}$ and $B$ have exactly the same singular values.
Let SVD of $B = U_B \Diag{\sigma(B)} V_B^{\top}$ where $\sigma(B) = [\sigma_1(B), \dots, \sigma_m(B)]$.
As $U$ and $V$ are orthogonal,
it is easy to see $\left(U U_B\right) \Diag{\sigma(B)} \left(V_B V\right)^{\top}$
is the SVD of $X$. 
Thus,
the Proposition holds.
\end{proof}


\subsection{Theorem~\ref{thm:conv:ncg}}
\label{app:thm:conv:ncg}

\begin{proof}
We first introduce two propositions.

\begin{proposition}\citep{mishra2013low} \label{pr:localopt}
For a square matrix $X$,
let $\sym{X} = \frac{1}{2}(X + X^{\top})$.
The first-order optimality conditions for \eqref{eq:local} are
\begin{align*}
\nabla \bar{f}(X) V B  - U \sym{U^{\top} \nabla \bar{f}(X) V B} & = 0, \\
(\nabla \bar{f}(X))^{\top} U B - V \sym{V^{\top} \nabla \bar{f}(X) U B} & = 0, \\
\sym{U^{\top} \nabla \bar{f}(X) V} + \bar{\mu} I & = 0.  \end{align*}
\end{proposition}

\begin{proposition} \label{pr:equivopt}
	If \eqref{eq:equiv:lowrank} has a critical point with rank-$r$,
	choose matrix size of $U \in \R^{m \times r}$, $V \in \R^{n \times r}$ and $B \in \mathcal{S}_+^{r \times r}$, 
	then any critical points of \eqref{eq:local} is also a critical point of \eqref{eq:equiv:lowrank}.
\end{proposition}

\begin{proof}
Subdifferential of the nuclear norm can be obtained as \citep{watson1992characterization}
	\begin{align}
	\partial \NM{X}{*}
	= \{ U V^{\top} + W : U^{\top} W = 0, W V = 0, \NM{W}{\infty} \le 1 \},
	\label{eq:defsubnn}
	\end{align}
	where $X = U B V^{\top}$. Let $\hat{X} = \hat{U} \hat{B} \hat{V}^{\top}$ be a critical point of \eqref{eq:local},
we have $\sym{\hat{U}^{\top} \nabla \bar{f}(\hat{X}) \hat{V}} + \bar{\mu} I = 0$ dues to Proposition~\ref{pr:localopt}.
	From property of matrix norm,
	we have
	\begin{align*}
	\lambda = \NM{\sym{\hat{U}^{\top} \nabla \bar{f}(\hat{X}) \hat{V}}}{\infty} 
	\le \NM{\hat{U}^{\top} \nabla \bar{f}(\hat{X}) \hat{V}}{\infty} \le \NM{\nabla \bar{f}(\hat{X})}{\infty}.
	\end{align*}
	The equality holds only when $\nabla \bar{f}(\hat{X}) = - \bar{\mu} \hat{U} \hat{V}^{\top} - \bar{\mu} \hat{U}_{\bot} \hat{\Sigma}_{\bot} \hat{V}^{\top}_{\bot}$
	where $\hat{U}_{\bot}$ and $\hat{V}_{\bot}$ are orthogonal matrix with $\hat{U}^{\top} \hat{U}_{\bot} = 0$ and $\hat{V}^{\top} \hat{V}_{\bot} = 0$, and $\hat{\Sigma}_{\bot}$ is a diagonal matrix with positive elements $[\Sigma_{\bot}]_{ii} \le 1$. 
	Combining this fact with \eqref{eq:defsubnn},
	we can see
	\begin{align}
	\nabla \bar{f}(\hat{X}) \in - \bar{\mu} \partial \NM{\hat{X}}{*}.
	\label{eq:temp20}
	\end{align}
	Then, for \eqref{eq:equiv:lowrank}, if $X_*$ is a critical point then we have
	\begin{align}
	\nabla \bar{f}(X_*) \in - \bar{\mu} \partial \NM{X_*}{*}.
	\label{eq:temp21}
	\end{align}
	Comparing \eqref{eq:temp20} and \eqref{eq:temp21}, 
	the difference is on rank of $\hat{X}$ and $X_*$.
	As \eqref{eq:equiv:lowrank} has a critical point with rank-$r$
	a critical point of \eqref{eq:local}, 
	$\hat{X}$ is also a critical point of \eqref{eq:equiv:lowrank}.
\end{proof}

In Algorithm~\ref{alg:novelfw},
the size of $U$, $V$ and $B$ are picked up as $m \times t$, $n \times t$ and $t \times t$.
If \eqref{eq:equiv:lowrank} has a critical point with rank-$r$,
then as iteration goes and $t = r$,  
from Proposition~\ref{pr:equivopt},
Algorithm~\ref{alg:novelfw} will return a critical point of \eqref{eq:equiv:lowrank}.
\end{proof}

\section{Details in Section~\ref{sec:imgden}}

\subsection{CCCP}
\label{sec:app:CCCP}

Using Proposition~\ref{pr:anotherDC},
we can decompose $\kappa(|x|) = \hat{\varsigma}(x) + \tilde{\varsigma}(x)$ where
$\hat{\varsigma}(x) = - \frac{\rho}{2} x^2$ is convex
and $\tilde{\varsigma}(x) = \kappa(|x|) + \frac{\rho}{2} x^2$ is concave.
We can apply above decomposition on $\kappa$ into \eqref{eq:ncvximgTV} and get
a DC decomposition as
\begin{align*}
& \tilde{F}(X)
= \sum_{i = 1}^m \sum_{j = 1}^n \tilde{\varsigma}\left( \left[ Y - X \right]_{ij} \right) 
+ \mu \sum_{i = 1}^{m - 1} \sum_{j = 1}^m \tilde{\varsigma}\left(\left[ D_v X \right]_{ij}\right)
+ \mu \sum_{i = 1}^{n} \sum_{j = 1}^{n - 1} \tilde{\varsigma}\left(\left[ X D_h \right]_{ij}\right),
\\
& \hat{F}(X)
= \sum_{i = 1}^m \sum_{j = 1}^n \hat{\varsigma}\left( \left[ Y - X \right]_{ij} \right) 
+ \mu \sum_{i = 1}^{m - 1} \sum_{j = 1}^m \hat{\varsigma}\left(\left[ D_v X \right]_{ij}\right)
+ \mu \sum_{i = 1}^{n} \sum_{j = 1}^{n - 1} \hat{\varsigma}\left(\left[ X D_h \right]_{ij}\right).
\end{align*}
Then, CCCP procedures at Section~\ref{sec:difcp} can be applied.

\subsection{Smoothing}
\label{sec:app:Smoothing}

As LSP function is used as $\kappa$, 
a smoothed version of it can be obtained as
$\tilde{\kappa}_{\lambda}(x) = \beta \log\left( 1 + \frac{h_{\lambda}(x)}{\theta} \right) $
where $h_{\lambda}(x)
= 
\begin{cases}
|x| & \text{if}\; |x| \ge \lambda
\\
\frac{x^2}{2\lambda} + \frac{\lambda}{2} & \text{otherwise}
\end{cases}$.
Thus, \eqref{eq:ncvximgTV} is smoothed as
\begin{align*}
\tilde{F}_{\lambda}(X)
= \sum_{i = 1}^m \sum_{j = 1}^n \tilde{\kappa}_{\lambda}\left( \left[Y - X\right]_{ij} \right) 
+ \mu \sum_{i = 1}^{m - 1} \sum_{j = 1}^m \tilde{\kappa}_{\lambda}\left( \left[ D_v X \right]_{ij} \right)
+ \mu \sum_{i = 1}^{n} \sum_{j = 1}^{n - 1} \tilde{\kappa}_{\lambda}\left( \left[ X D_h \right]_{ij} \right).
\end{align*}
Then, gradient descent is used for optimization \citep{chen2012smoothing}.
Specifically, 
we need to minimize a sequence of subproblems $\{\tilde{F}_{\lambda_1}(X), \tilde{F}_{\lambda_2}(X), \dots \}$
with $\lambda_i = \lambda_0 \cdot \nu^i$,
and using $X$ from $\tilde{F}_{\lambda_{i - 1}}(X)$ to warm start $\tilde{F}_{\lambda_i}(X)$.
In the experiment, we set $\lambda_0 = 0.1$ and $\nu = 0.95$.

{
\bibliography{bib}

\begin{thebibliography}{79}
\providecommand{\natexlab}[1]{#1}
\providecommand{\url}[1]{\texttt{#1}}
\expandafter\ifx\csname urlstyle\endcsname\relax
  \providecommand{\doi}[1]{doi: #1}\else
  \providecommand{\doi}{doi: \begingroup \urlstyle{rm}\Url}\fi

\bibitem[Andrew and Gao(2007)]{andrew2007scalable}
G.~Andrew and J.~Gao.
\newblock Scalable training of $\ell_1$-regularized log-linear models.
\newblock In \emph{Proceedings of the 24th International Conference on Machine
  learning}, pages 33--40, 2007.

\bibitem[Bauschke et~al.(2008)Bauschke, Goebel, Lucet, and
  Wang]{bauschke2008proximal}
H.~Bauschke, R.~Goebel, Y.~Lucet, and X.~Wang.
\newblock The proximal average: Basic theory.
\newblock \emph{SIAM Journal on Optimization}, 19\penalty0 (2):\penalty0
  766--785, 2008.

\bibitem[Beck and Teboulle(2009)]{beck2009fasttv}
A.~Beck and M.~Teboulle.
\newblock Fast gradient-based algorithms for constrained total variation image
  denoising and deblurring problems.
\newblock \emph{IEEE Transactions on Image Processing}, 18\penalty0
  (11):\penalty0 2419--2434, 2009.

\bibitem[Beck(2009)]{beck2009fast}
M.~Beck, A.and~Teboulle.
\newblock A fast iterative shrinkage-thresholding algorithm for linear inverse
  problems.
\newblock \emph{SIAM Journal on Imaging Sciences}, 2\penalty0 (1):\penalty0
  183--202, 2009.

\bibitem[Bertsekas(1999)]{bertsekas1999nonlinear}
D.P. Bertsekas.
\newblock \emph{Nonlinear Programming}.
\newblock Athena Scientific, 1999.

\bibitem[Bertsekas and Tsitsiklis(1989)]{Bertsekas1989}
D.P. Bertsekas and J.N. Tsitsiklis.
\newblock \emph{Parallel and Distributed Computation: Numerical Methods}.
\newblock Prentice-Hall, 1989.

\bibitem[Bot et~al.(2016)Bot, Csetnek, and L{\'a}szl{\'o}]{boct2016inertial}
R.I. Bot, E.~Robert Csetnek, and S.C. L{\'a}szl{\'o}.
\newblock An inertial forward-backward algorithm for the minimization of the
  sum of two nonconvex functions.
\newblock \emph{{EURO} Journal on Computational Optimization}, 4\penalty0
  (1):\penalty0 3--25, 2016.

\bibitem[Bottou(1998)]{bottou1998online}
L.~Bottou.
\newblock Online learning and stochastic approximations.
\newblock \emph{On-line Learning in Neural Networks}, 17\penalty0 (9):\penalty0
  142--174, 1998.

\bibitem[Boyd and Vandenberghe(2004)]{boyd2004convex}
S.~Boyd and L.~Vandenberghe.
\newblock \emph{Convex Optimization}.
\newblock Cambridge University Press, 2004.

\bibitem[Boyd et~al.(2011)Boyd, Parikh, Chu, Peleato, and
  Eckstein]{boyd2011distributed}
S.~Boyd, N.~Parikh, E.~Chu, B.~Peleato, and J.~Eckstein.
\newblock Distributed optimization and statistical learning via the alternating
  direction method of multipliers.
\newblock \emph{Foundations and Trends in Machine Learning}, 3\penalty0
  (1):\penalty0 1--122, 2011.

\bibitem[Bredies et~al.(2009)Bredies, Lorenz, and
  Maass]{bredies2009generalized}
K.~Bredies, D.A. Lorenz, and P.~Maass.
\newblock A generalized conditional gradient method and its connection to an
  iterative shrinkage method.
\newblock \emph{Computational Optimization and Applications}, 42\penalty0
  (2):\penalty0 173--193, 2009.

\bibitem[Cand{\`e}s and Recht(2009)]{candes2009exact}
E.J. Cand{\`e}s and B.~Recht.
\newblock Exact matrix completion via convex optimization.
\newblock \emph{Foundations of Computational Mathematics}, 9\penalty0
  (6):\penalty0 717--772, 2009.

\bibitem[Cand{\`e}s et~al.(2008)Cand{\`e}s, Wakin, and
  Boyd]{candes2008enhancing}
E.J. Cand{\`e}s, M.B. Wakin, and S.~Boyd.
\newblock Enhancing sparsity by reweighted $\ell_1$ minimization.
\newblock \emph{Journal of Fourier Analysis and Applications}, 14\penalty0
  (5-6):\penalty0 877--905, 2008.

\bibitem[Cand\`{e}s et~al.(2011)Cand\`{e}s, Li, Ma, and
  Wright]{candes2011robust}
E.J. Cand\`{e}s, X.~Li, Yi~Ma, and J.~Wright.
\newblock Robust principal component analysis.
\newblock \emph{Journal of the ACM}, 58\penalty0 (3):\penalty0 1--37, 2011.

\bibitem[Chen(2012)]{chen2012smoothing}
X.~Chen.
\newblock Smoothing methods for nonsmooth, nonconvex minimization.
\newblock \emph{Mathematical Programming}, 134\penalty0 (1):\penalty0 71--99,
  2012.

\bibitem[Dabov et~al.(2007)Dabov, Foi, Katkovnik, and
  Egiazarian]{dabov2007image}
K.~Dabov, A.~Foi, V.~Katkovnik, and K.~Egiazarian.
\newblock Image denoising by sparse {3-D} transform-domain collaborative
  filtering.
\newblock \emph{IEEE Transactions on Image Processing}, 16\penalty0
  (8):\penalty0 2080--2095, 2007.

\bibitem[Donoho(2006)]{donoho2006compressed}
D.L. Donoho.
\newblock Compressed sensing.
\newblock \emph{IEEE Transactions on Information Theory}, 52\penalty0
  (4):\penalty0 1289--1306, 2006.

\bibitem[Eriksson et~al.(2004)Eriksson, Estep, and
  Johnson]{eriksson2013applied}
K.~Eriksson, F.~Estep, and C.~Johnson.
\newblock \emph{Applied Mathematics: Body and Soul: Volume 1: Derivatives and
  Geometry in IR3}.
\newblock Springer-Verlag, 2004.

\bibitem[Fan and Li(2001)]{fan2001variable}
J.~Fan and R.~Li.
\newblock Variable selection via nonconcave penalized likelihood and its oracle
  properties.
\newblock \emph{Journal of the American Statistical Association}, 96\penalty0
  (456):\penalty0 1348--1360, 2001.

\bibitem[Frank and Wolfe(1956)]{frank1956algorithm}
M.~Frank and P.~Wolfe.
\newblock An algorithm for quadratic programming.
\newblock \emph{Naval Research Logistics}, 3\penalty0 (1-2):\penalty0 95--110,
  1956.

\bibitem[Geman and Yang(1995)]{geman1995nonlinear}
D.~Geman and C.~Yang.
\newblock Nonlinear image recovery with half-quadratic regularization.
\newblock \emph{IEEE Transactions on Image Processing}, 4\penalty0
  (7):\penalty0 932--946, 1995.

\bibitem[Ghadimi and Lan(2016)]{ghadimi2016accelerated}
S.~Ghadimi and G.~Lan.
\newblock Accelerated gradient methods for nonconvex nonlinear and stochastic
  programming.
\newblock \emph{Mathematical Programming}, 156\penalty0 (1-2):\penalty0 59--99,
  2016.

\bibitem[Glowinski and Marroco(1975)]{glowinski1975}
R.~Glowinski and A.~Marroco.
\newblock Sur l'approximation, par {\'e}l{\'e}ments finis d'ordre un, et la
  r{\'e}solution, par p{\'e}nalisation-dualit{\'e} d'une classe de
  probl{\`e}mes de dirichlet non lin{\'e}aires.
\newblock \emph{Revue fran{\c{c}}aise d'automatique, informatique, recherche
  op{\'e}rationnelle. Analyse num{\'e}rique}, 9\penalty0 (2):\penalty0 41--76,
  1975.

\bibitem[Golub and Van~Loan(2012)]{golub2012matrix}
G.H. Golub and C.F. Van~Loan.
\newblock \emph{Matrix Computations}.
\newblock Johns Hopkins University Press, 2012.

\bibitem[Gong and Ye(2015{\natexlab{a}})]{gong2015honor}
P.~Gong and J.~Ye.
\newblock {HONOR}: Hybrid {O}ptimization for {NO}n-convex {R}egularized
  problems.
\newblock In \emph{Advances in Neural Information Processing Systems}, pages
  415--423, 2015{\natexlab{a}}.

\bibitem[Gong and Ye(2015{\natexlab{b}})]{gong2015modified}
P.~Gong and J.~Ye.
\newblock A modified orthant-wise limited memory quasi-{N}ewton method with
  convergence analysis.
\newblock In \emph{Proceedings of the 32nd International Conference on Machine
  Learning}, pages 276--284, 2015{\natexlab{b}}.

\bibitem[Gong et~al.(2013)Gong, Zhang, Lu, Huang, and Ye]{gongZLHY2013}
P.~Gong, C.~Zhang, Z.~Lu, J.~Huang, and J.~Ye.
\newblock A general iterative shrinkage and thresholding algorithm for
  non-convex regularized optimization problems.
\newblock In \emph{Proceedings of the 30th International Conference on Machine
  Learning}, pages 37--45, 2013.

\bibitem[Gui et~al.(2016)Gui, Han, and Gu]{goqqlow2016}
H.~Gui, J.~Han, and Q.~Gu.
\newblock Towards faster rates and oracle property for low-rank matrix
  estimation.
\newblock In \emph{Proceedings of the 33nd International Conference on Machine
  Learning}, pages 2300--2309, 2016.

\bibitem[He and Yuan(2012)]{he20121}
B.~He and X.~Yuan.
\newblock On the $o(1/n)$ convergence rate of the douglas-rachford alternating
  direction method.
\newblock \emph{SIAM Journal on Numerical Analysis}, 50\penalty0 (2):\penalty0
  700--709, 2012.

\bibitem[Hiriart-Urruty(1985)]{hiriart85}
J.B. Hiriart-Urruty.
\newblock Generalized differentiability, duality and optimization for problems
  dealing with differences of convex functions.
\newblock \emph{Convexity and Duality in Optimization}, pages 37--70, 1985.

\bibitem[Hong et~al.(2016)Hong, Luo, and Razaviyayn]{hong2016convergence}
M.~Hong, Z.-Q. Luo, and M.~Razaviyayn.
\newblock Convergence analysis of alternating direction method of multipliers
  for a family of nonconvex problems.
\newblock \emph{SIAM Journal on Optimization}, 26\penalty0 (1):\penalty0
  337--364, 2016.

\bibitem[Hsieh and Olsen(2014)]{hsieh2014nuclear}
C.-J. Hsieh and P.~Olsen.
\newblock Nuclear norm minimization via active subspace selection.
\newblock In \emph{Proceedings of the 31st International Conference on Machine
  Learning}, pages 575--583, 2014.

\bibitem[Jacob et~al.(2009)Jacob, Obozinski, and Vert]{jacob2009group}
L.~Jacob, G.~Obozinski, and J.-P. Vert.
\newblock Group lasso with overlap and graph lasso.
\newblock In \emph{Proceedings of the 26th International Conference on Machine
  Learning}, pages 433--440, 2009.

\bibitem[Jaggi(2013)]{jaggi2013revisiting}
M.~Jaggi.
\newblock Revisiting {F}rank-{W}olfe: Projection-free sparse convex
  optimization.
\newblock In \emph{Proceedings of the 30th International Conference on Machine
  Learning}, pages 427--435, 2013.

\bibitem[Jenatton et~al.(2011)Jenatton, Mairal, Obozinski, and
  Bach]{jenatton2011proximal}
R.~Jenatton, J.~Mairal, G.~Obozinski, and F.~Bach.
\newblock Proximal methods for hierarchical sparse coding.
\newblock \emph{Journal of Machine Learning Research}, 12:\penalty0 2297--2334,
  2011.

\bibitem[Johnson and Zhang(2013)]{johnson2013accelerating}
R.~Johnson and T.~Zhang.
\newblock Accelerating stochastic gradient descent using predictive variance
  reduction.
\newblock In \emph{Advances in Neural Information Processing Systems}, pages
  315--323, 2013.

\bibitem[Laue(2012)]{laue2012hybrid}
S.~Laue.
\newblock A hybrid algorithm for convex semidefinite optimization.
\newblock In \emph{Proceedings of the 29th International Conference on Machine
  Learning}, pages 177--184, 2012.

\bibitem[Lewis and Sendov(2005)]{lewis2005nonsmooth}
A.S. Lewis and H.S. Sendov.
\newblock Nonsmooth analysis of singular values. {P}art ii: Applications.
\newblock \emph{Set-Valued Analysis}, 13\penalty0 (3):\penalty0 243--264, 2005.

\bibitem[Li and Pong(2015)]{li2015global}
G.~Li and T.K. Pong.
\newblock Global convergence of splitting methods for nonconvex composite
  optimization.
\newblock \emph{SIAM Journal on Optimization}, 25\penalty0 (4):\penalty0
  2434--2460, 2015.

\bibitem[Li and Lin(2015)]{li2015accelerated}
H.~Li and Z.~Lin.
\newblock Accelerated proximal gradient methods for nonconvex programming.
\newblock In \emph{Advances in Neural Information Processing Systems}, pages
  379--387, 2015.

\bibitem[Liu and Ye(2010)]{liu2010moreau}
J.~Liu and J.~Ye.
\newblock Moreau-{Y}osida regularization for grouped tree structure learning.
\newblock In \emph{Advances in Neural Information Processing Systems}, pages
  1459--1467, 2010.

\bibitem[Liu et~al.(2013)Liu, Musialski, Wonka, and Ye]{liu2013tensor}
J.~Liu, P.~Musialski, P.~Wonka, and J.~Ye.
\newblock Tensor completion for estimating missing values in visual data.
\newblock \emph{IEEE Transactions on Pattern Analysis and Machine
  Intelligence}, 35\penalty0 (1):\penalty0 208--220, 2013.

\bibitem[Lu et~al.(2013)Lu, Shi, and Jia]{lu2013online}
C.~Lu, J.~Shi, and J.~Jia.
\newblock Online robust dictionary learning.
\newblock In \emph{IEEE Conference on Computer Vision and Pattern Recognition},
  pages 415--422, 2013.

\bibitem[Lu et~al.(2014)Lu, Tang, Yan, and Lin]{lu2014generalized}
C.~Lu, J.~Tang, S.~Yan, and Z.~Lin.
\newblock Generalized nonconvex nonsmooth low-rank minimization.
\newblock In \emph{Proceedings of the International Conference on Computer
  Vision and Pattern Recognition}, pages 4130--4137, 2014.

\bibitem[Lu et~al.(2015)Lu, Zhu, Xu, Yan, and Lin]{lu2015generalized}
C.~Lu, C.~Zhu, C.~Xu, S.~Yan, and Z.~Lin.
\newblock Generalized singular value thresholding.
\newblock In \emph{Proceedings of the 29th AAAI Conference on Artificial
  Intelligence}, pages 1805--1811, 2015.

\bibitem[Lu(2012)]{zhaosong2012}
Z.~Lu.
\newblock Sequential convex programming methods for a class of structured
  nonlinear programming.
\newblock Preprint arXiv:1210.3039, 2012.

\bibitem[Mairal et~al.(2009)Mairal, Bach, Ponce, and Sapiro]{mairal2009online}
J.~Mairal, F.~Bach, J.~Ponce, and G.~Sapiro.
\newblock Online dictionary learning for sparse coding.
\newblock In \emph{Proceedings of the 26th International Conference on Machine
  Learning}, pages 689--696, 2009.

\bibitem[Mazumder et~al.(2010)Mazumder, Hastie, and
  Tibshirani]{mazumder2010spectral}
R.~Mazumder, T.~Hastie, and R.~Tibshirani.
\newblock Spectral regularization algorithms for learning large incomplete
  matrices.
\newblock \emph{Journal of Machine Learning Research}, 11:\penalty0 2287--2322,
  2010.

\bibitem[Mishra et~al.(2013)Mishra, Meyer, Bach, and Sepulchre]{mishra2013low}
B.~Mishra, G.~Meyer, F.~Bach, and R.~Sepulchre.
\newblock Low-rank optimization with trace norm penalty.
\newblock \emph{SIAM Journal on Optimization}, 23\penalty0 (4):\penalty0
  2124--2149, 2013.

\bibitem[Nesterov(2013)]{nesterov2013gradient}
Y.~Nesterov.
\newblock Gradient methods for minimizing composite functions.
\newblock \emph{Mathematical Programming}, 140\penalty0 (1):\penalty0 125--161,
  2013.

\bibitem[Ngo and Saad(2012)]{ngo2012scaled}
T.~Ngo and Y.~Saad.
\newblock Scaled gradients on {G}rassmann manifolds for matrix completion.
\newblock In \emph{Advances in Neural Information Processing Systems}, pages
  1412--1420, 2012.

\bibitem[Nikolova(2004)]{nikolova2004variational}
M.~Nikolova.
\newblock A variational approach to remove outliers and impulse noise.
\newblock \emph{Journal of Mathematical Imaging and Vision}, 20\penalty0
  (1-2):\penalty0 99--120, 2004.

\bibitem[Nocedal and Wright(2006)]{nocedal2006numerical}
J.~Nocedal and S.J. Wright.
\newblock \emph{Numerical Optimization}.
\newblock Springer, 2006.

\bibitem[Ochs et~al.(2014)Ochs, Chen, Brox, and Pock]{ochs2014ipiano}
P.~Ochs, Y.~Chen, T.~Brox, and T.~Pock.
\newblock i{P}iano: Inertial proximal algorithm for nonconvex optimization.
\newblock \emph{SIAM Journal on Imaging Sciences}, 7\penalty0 (2):\penalty0
  1388--1419, 2014.

\bibitem[Parikh and Boyd(2013)]{parikh2013proximal}
N.~Parikh and S.~Boyd.
\newblock Proximal algorithms.
\newblock \emph{Foundations and Trends in Optimization}, 1\penalty0
  (3):\penalty0 123--231, 2013.

\bibitem[Reddi et~al.(2016{\natexlab{a}})Reddi, Hefny, Sra, P{\'{o}}czos, and
  Smola]{reddiHSPS2016}
S.J. Reddi, A.~Hefny, S.~Sra, B.~P{\'{o}}czos, and A.J. Smola.
\newblock Stochastic variance reduction for nonconvex optimization.
\newblock In \emph{Proceedings of the 33nd International Conference on Machine
  Learning}, pages 314--323, 2016{\natexlab{a}}.

\bibitem[Reddi et~al.(2016{\natexlab{b}})Reddi, Sra, Poczos, and
  Smola]{reddi2016fast}
S.J. Reddi, S.~Sra, B.~Poczos, and A.~Smola.
\newblock Fast stochastic methods for nonsmooth nonconvex optimization.
\newblock In \emph{Advances in Neural Information Processing Systems}, pages
  1145--1153, 2016{\natexlab{b}}.

\bibitem[Schmidt et~al.(2011)Schmidt, Roux, and Bach]{schmidt2011convergence}
M.~Schmidt, N.L. Roux, and F.~Bach.
\newblock Convergence rates of inexact proximal-gradient methods for convex
  optimization.
\newblock In \emph{Advances in Neural Information Processing Systems}, pages
  1458--1466, 2011.

\bibitem[Sra(2012)]{sra2012scalable}
S.~Sra.
\newblock Scalable nonconvex inexact proximal splitting.
\newblock In \emph{Advances in Neural Information Processing Systems}, pages
  530--538, 2012.

\bibitem[Srebro et~al.(2004)Srebro, Rennie, and Jaakkola]{srebro2004maximum}
N.~Srebro, J.~Rennie, and T.S. Jaakkola.
\newblock Maximum-margin matrix factorization.
\newblock In \emph{Advances in Neural Information Processing Systems}, pages
  1329--1336, 2004.

\bibitem[Sun et~al.(2013)Sun, Xiang, and Ye]{sun2013robust}
Q.~Sun, S.~Xiang, and J.~Ye.
\newblock Robust principal component analysis via capped norms.
\newblock In \emph{Proceedings of the 19th International Conference on
  Knowledge Discovery and Data Mining}, pages 311--319, 2013.

\bibitem[Tibshirani(1996)]{tibshirani1996regression}
R.~Tibshirani.
\newblock Regression shrinkage and selection via the lasso.
\newblock \emph{Journal of the Royal Statistical Society. Series B},
  73\penalty0 (3):\penalty0 273--282, 1996.

\bibitem[Tibshirani et~al.(2005)Tibshirani, Saunders, Rosset, Zhu, and
  Knight]{tibshirani2005sparsity}
R.~Tibshirani, M.~Saunders, S.~Rosset, J.~Zhu, and K.~Knight.
\newblock Sparsity and smoothness via the fused lasso.
\newblock \emph{Journal of the Royal Statistical Society: Series B},
  67\penalty0 (1):\penalty0 91--108, 2005.

\bibitem[Trzasko and Manduca(2009)]{trzasko2009highly}
J.~Trzasko and A.~Manduca.
\newblock Highly undersampled magnetic resonance image reconstruction via
  homotopic-minimization.
\newblock \emph{IEEE Transactions on Medical Imaging}, 28\penalty0
  (1):\penalty0 106--121, 2009.

\bibitem[Watson(1992)]{watson1992characterization}
G.A. Watson.
\newblock Characterization of the subdifferential of some matrix norms.
\newblock \emph{Linear Algebra and its Applications}, 170:\penalty0 33--45,
  1992.

\bibitem[Wen et~al.(2012)Wen, Yin, and Zhang]{wen2012solving}
Z.~Wen, W.~Yin, and Y.~Zhang.
\newblock Solving a low-rank factorization model for matrix completion by a
  nonlinear successive over-relaxation algorithm.
\newblock \emph{Mathematical Programming Computation}, 4\penalty0 (4):\penalty0
  333--361, 2012.

\bibitem[Xiao and Zhang(2014)]{xiao2014proximal}
L.~Xiao and T.~Zhang.
\newblock A proximal stochastic gradient method with progressive variance
  reduction.
\newblock \emph{SIAM Journal on Optimization}, 24\penalty0 (4):\penalty0
  2057--2075, 2014.

\bibitem[Yan(2013)]{yan2013restoration}
M.~Yan.
\newblock Restoration of images corrupted by impulse noise and mixed gaussian
  impulse noise using blind inpainting.
\newblock \emph{SIAM Journal on Imaging Sciences}, 6\penalty0 (3):\penalty0
  1227--1245, 2013.

\bibitem[Yang et~al.(2011)Yang, Zhang, Yang, and Zhang]{yang2011robust}
M.~Yang, L.~Zhang, J.~Yang, and D.~Zhang.
\newblock Robust sparse coding for face recognition.
\newblock In \emph{IEEE Conference on Computer Vision and Pattern Recognition},
  pages 625--632, 2011.

\bibitem[Yao and Kwok(2016)]{yao2016efficient}
Q~Yao and J.~T. Kwok.
\newblock Efficient learning with a family of nonconvex regularizers by
  redistributing nonconvexity.
\newblock In \emph{Proceedings of the 33rd International Conference on Machine
  Learning}, pages 2645--2654, 2016.

\bibitem[Yao et~al.(2015)Yao, Kwok, and Zhong]{qyao2015icdm}
Q.~Yao, J.T. Kwok, and W.~Zhong.
\newblock Fast low-rank matrix learning with nonconvex regularization.
\newblock In \emph{Proceedings of IEEE International Conference on Data
  Mining}, pages 539--548, 2015.

\bibitem[Yuan et~al.(2011)Yuan, Liu, and Ye]{yuan2011efficient}
L.~Yuan, J.~Liu, and J.~Ye.
\newblock Efficient methods for overlapping group lasso.
\newblock In \emph{Advances in Neural Information Processing Systems}, pages
  352--360, 2011.

\bibitem[Yuille and Rangarajan(2002)]{yuille2002concave}
A.L. Yuille and A.~Rangarajan.
\newblock The concave-convex procedure ({CCCP}).
\newblock In \emph{Advances in Neural Information Processing Systems}, pages
  1033--1040, 2002.

\bibitem[Zhang(2010{\natexlab{a}})]{zhang2010nearly}
C.H. Zhang.
\newblock Nearly unbiased variable selection under minimax concave penalty.
\newblock \emph{Annals of Statistics}, 38\penalty0 (2):\penalty0 894--942,
  2010{\natexlab{a}}.

\bibitem[Zhang(2010{\natexlab{b}})]{zhang2010analysis}
T.~Zhang.
\newblock Analysis of multi-stage convex relaxation for sparse regularization.
\newblock \emph{Journal of Machine Learning Research}, 11:\penalty0 1081--1107,
  2010{\natexlab{b}}.

\bibitem[Zhang et~al.(2012)Zhang, Schuurmans, and Yu]{zhang2012accelerated}
X.~Zhang, D.~Schuurmans, and Y.-L. Yu.
\newblock Accelerated training for matrix-norm regularization: A boosting
  approach.
\newblock In \emph{Advances in Neural Information Processing Systems}, pages
  2906--2914, 2012.

\bibitem[Zhao et~al.(2011)Zhao, Wang, and Cham]{zhao2011background}
C.~Zhao, X.~Wang, and W.-K. Cham.
\newblock Background subtraction via robust dictionary learning.
\newblock \emph{EURASIP Journal on Image and Video Processing}, 2011\penalty0
  (1):\penalty0 1--12, 2011.

\bibitem[Zhong and Kwok(2014)]{zhongK2014gdpan}
W.~Zhong and J.T. Kwok.
\newblock Gradient descent with proximal average for nonconvex and composite
  regularization.
\newblock In \emph{Proceedings of the 28th AAAI Conference on Artificial
  Intelligence}, pages 2206--2212, 2014.

\bibitem[Zhu and Hazan(2016)]{zhuH2016}
Z.A. Zhu and E.~Hazan.
\newblock Variance reduction for faster non-convex optimization.
\newblock In \emph{Proceedings of the 33nd International Conference on Machine
  Learning}, pages 699--707, 2016.

\end{thebibliography}
}

\end{document}